\documentclass[a4paper,10pt]{article}

\usepackage{amssymb,amsmath,amsthm}
\usepackage{hyperref}
\usepackage{latexsym}
\usepackage{amsmath}
\usepackage{amsthm,color}
\usepackage{amssymb}
\usepackage{mathrsfs}
\usepackage{xfrac}

\usepackage{placeins}

\newtheorem{theorem}{Theorem}[section]
\newtheorem{lemma}[theorem]{Lemma}
\newtheorem{definition}[theorem]{Definition}
\newtheorem{assumption}[theorem]{Assumption}
\newtheorem{prop}[theorem]{Proposition}
\newtheorem{remark}[theorem]{Remark}

\newcommand{\R}{\mathbb{R}}
\newcommand{\E}{\mathbb{E}}
\title{Additive noise destroys the random attractor close to bifurcation}
\author{Luigi A. Bianchi, Dirk Bl\"omker\\  (Universit\"at Augsburg, Germany),  \\
\& \\
Meihua Yang \\ (Huazhong University of Science and Technology, Wuhan, China)}
\date{\today}

\begin{document}

\maketitle

\begin{abstract}
We provide an example for stabilization by noise. Our approach does not rely on monotonicity arguments due to the presence of higher order differential operators
or mixing properties of the system as the noise might be highly degenerate.

In the examples a
scalar additive noise destroys a high-dimensional random attractor of a PDE on an unbounded domain.
In the presence of small noise close to bifurcation all trajectories converge to a  single stationary solution.

\end{abstract}

\tableofcontents


\section{Introduction}

Our main aim is to provide another example for the stabilization
of a random dynamical system due to noise, which does not seem to fit in the many classes considered before.
We do not rely on monotonicity, as the differential operators are higher order. 
Furthermore, the mixing properties used in other approaches might fail, as the noise is highly degenerate.
We focus in the examples on equations of the following type
\[
\partial_t u = A u + \nu u  + f(u) + \sigma \partial_t\beta,
\]
posed on the whole domain $\mathbb{R^d}$, $d\geq1$, where we do not assume any decay condition at infinity.
The operator $A$ is a non-positive polynomial of the Laplacian with non-empty kernel.
Thus the scalar $\nu$ is an explicit measure for the distance from bifurcation. 

The nonlinearity is a function of $u$ with the property that $f(u) \sim -|u|^{p}u$ for large $u$ and some $p>0$,
with the standard example being a stable cubic~$-u^3$.
For the noise we assume that it is spatially constant and thus
given by a single standard real valued Brownian motion $\beta$. 
We comment on other types of additive noise later on. 

One standard  example is the stochastic Swift-Hohenberg with scalar additive noise, 
where $A=-(1+\Delta)^2$ and $f(u)=-u^3$. 

For our main result on stabilization, we exploit directly the structure of the equation and consider
noise sufficiently larger 
than the distance from bifurcation, which means for Swift-Hohenberg $\nu<\frac32\sigma^2$.
In that case, we show that any stationary solution is already globally stable.
Due to our special setting, the stationary solution is spatially constant, and solves an SDE. 
Thus by the result of  Crauel \& Flandoli \cite{crafla98} it is unique among the spatially constant ones, 
because of the monotonicity of the one-dimensional system.
In conclusion, we cannot show that arbitrarily small additive noise
destroys a pitchfork-bifurcation, but at least it shifts the bifurcation.

In the general setting, as well as in the example of Swift-Hohenberg, we cannot rely on
monotonicity like in the Allen-Cahn equation considered for instance by
Caraballo, Crauel, Langa \& Robinson in~\cite{carcralanrob07}. See also the general results of
 Chueshov \& Vuillermot, \cite{chuvui04},
or   Arnold \& Chueshov, \cite{arnchu98}, or  Flandoli, Gess \& Scheutzow, \cite{flagessch15}, which is a
more recent result based on monotonicity. Here due to the ordering of solutions, 
one can construct two ergodic stationary solution that stay ordered which leads to a contradiction.
Nevertheless, we use this result, when we construct spatially constant stationary solutions, as they satisfy an SDE.

Let us also mention  the work of  Lamb, Rasmussen et.\ al.,~\cite{caldoalamras13}, 
which shows that additive noise does not destroy a pitchfork-bifurcation, 
because a phenomenological bifurcation is still present, which can be seen particularly well where the noise is bounded (since in that case, if the random attractor is non trivial we can actually see it), a case that we do not consider here.

A different approach to stabilization was considered by Tearne in~\cite{tea08}, 
where he showed that sufficiently
small non-degenerate noise in a gradient system given by an ODE, 
leads to a  trivial random attractor and thus stabilization.

This was recently improved by   Flandoli,  Gess \& Scheutzow~\cite{flagessch14}
to general SDE systems without any restriction on the noise. 
They rely mainly on mixing properties, local stability and ``contraction in the large'',
in order to show that the weak random attractor is a single point. Unfortunately,
in our examples we cannot easily resort to a mixing property, as the noise in our examples is highly degenerate.
We can neither prove nor disprove a mixing property for solutions.

A more qualitative approach to stabilization using amplitude equations was presented by  Klepel, Mohammed \& Bl\"omker in~\cite{mohblokle13}.
For the Swift-Hohenberg equation with $\nu=\mathcal{O}(\varepsilon^2)$, $\sigma =\mathcal{O}(\varepsilon)$ and $\nu<\frac32\sigma^2$
they showed that, for sufficiently small $\varepsilon$, the dynamics is qualitatively described by a stable deterministic ODE.
This is the same condition we consider in our paper here, but we do not need any asymptotics for $\varepsilon\to0$, and we can consider large $\sigma$ and $\nu$.
Moreover, we present here an almost sure result for $t\to\infty$, while~\cite{mohblokle13} stated stabilization 
only with high probability and on finite time-scales of order $\varepsilon^{-2}$.

Let us also mention the 
upper semi-continuity of random attractors in the limit noise to zero (i.e.\ $\sigma\to0$).
This does not apply here, as in all the examples provided we can only study the joint limit $\sigma\to0$ and $\nu\to0$, 
in order to stay in the regime where the stabilization holds.

The existence of pull-back random attractors has been extensive studied
by many authors for several kinds of SPDEs defined on bounded domain (see, e.g.
\cite{SR1996attractors,PRSLA2007attractors,AAA2013attractors}). 
For the unbounded domain case, the situation becomes much more complicated, and we need to deal with more difficulties
caused by the unboundedness of the spacial domain. In an unbounded
domain, typical Sobolev embeddings are not compact and the spaces
$L^p(\mathbb{R}^N)$ are not nested, adding to the technical difficulties. Bates, Lu \& Wang~\cite{JDE2009attractors} introduced the tail estimates' technique to obtain compactness, 
and the existence of pull-back random attractor has been studied by the same authors in several publication, e.g.~\cite{NA2009attractors,DCDS2015attractors,JMAA2011attractors,JDE2012attractors}.
Moreover, the usual Sobolev spaces do not include the constant functions and travelling waves, so to be able to include these special solutions (e.g. equilibria and relaxation waves) in the attractor, 
they consider in~\cite{batluwan13} weighted spaces.

The results of this paper concern the  existence of random attractors.
For that purpose in Section \ref{sec:RA} we follow the ideas of the already cited~\cite{batluwan13}, where the authors show the existence of
tempered random attractors at least for parabolic equations of second order in weighted spaces.
In a short Section \ref{sec:SoA} we discuss the size of the deterministic attractor, which is high-dimensional. 
The key stabilization result is presented in Section \ref{sec:Stab} in Theorem \ref{thm:stab} 
where we also discuss the explicit examples and more general noise terms (see Theorem \ref{thm:genstab}). 

In the first part of the paper, we begin providing the definition of spaces and some key technical estimates in Section~\ref{sec:SaE}, and  the basic setting of the problem, the definition of random dynamical systems and attractors
as well as the properties of spatially constant stationary solutions in Section~\ref{sec:Set}.

\section{Spaces and Estimates}
\label{sec:SaE}

We consider for $\rho>0$ and $p\geq1$ the  weighted spaces $L^p_\rho$ with norm
\[
\|u\|_{L^p_\rho}^p = \int_\R \rho(x)|u(x)|^p dx \;, 
\]
for a sufficiently strong polynomial weight of the type
\[\rho(x)=(1+|cx|^2)^{-\rho/2}\quad\text{ and some small }c>0.
\]
Note that by a slight abuse of notation, we identify the weight $\rho:\R\to(0,1]$ 
with the scalar decay exponent $\rho>0$. An important feature of these spaces is that $L^p_\rho$ contains not only constant functions, but also functions unbounded at infinity.

A simple calculation verifies the following properties of the weight:
\[
|\rho'(x)| \leq C_1  \rho(x)
\quad\text{and}\quad
|\rho^{(n)}(x)| \leq C_n \rho(x)\;.
\]
In these estimates we can make the constants $C_n$ as small as we want, by choosing $c$ close to $0$, as $C_n\sim c^n$.
Moreover, it is easy to check that the weight is integrable (i.e.\ $\rho\in L^1(\R)$) if and only if $\rho>1$.

We will also use the $H^k_\rho$ spaces of functions with square integrable derivatives up to order $k$. 
The norm in these spaces is given by the sum of all $L^2_\rho$ norms of the function and all derivatives up to order $k$. 
Note that by the structure of the weight $u\in H^k_\rho$ if and only if $\sqrt{\rho}\cdot u\in H^k(\R)$.

For our examples we rely on the following estimate on the supremum of the numerical range of the Swift-Hohenberg operator in  $L^2_\rho$.
See also \cite{mielke1995attractors} or \cite{biablosch16}.

\begin{lemma}
\label{lem:spec} 
For any $\eta_0>0$ there is a choice of $c$ in the weight $\rho$ such that
\[\langle v , -(1+\partial_x^2)^2 v\rangle_{L^2_\rho} \leq - \eta_0 \|v''\|_{L^2_\rho}^2 + C\eta_0 \|v\|_{L^2_\rho}^2  \]
for all  $v\in H^2_\rho$. We can choose $\eta_0$ proportional to $c^2$, and $C>0$ is a universal constant.
\end{lemma}
\begin{proof}
First we prove the estimate for a function $v$ which is compactly supported and smooth.
Then we extend the result by continuity to $H^2_\rho$, where we use the fact that the bilinear form of $-(1+\partial_x^2)^2$
can be rewritten (using integration by parts) to depend only on second derivatives.

Thanks to integration by parts and H\"older's inequality, we obtain
\begin{multline*}
  \int_\R\rho v [-(1+\partial_x^2)^2] v\, dx\\
  \begin{aligned}
  &= -  \| v\|^2_{L^2_\rho} - 2 \int_\R\rho v v''\, dx +   \int_\R  \rho' v v''' \, dx +   \int_\R  \rho v' v''' \, dx\\
  &=  -    \| v\|^2_{L^2_\rho} - 2 \int_\R\rho v v''\, dx -   \int_\R  \rho'' v v'' \, dx -2 \int_\R  \rho' v' v'' \, dx -  \| v''\|^2_{L^2_\rho}  \\
    &=  -   \| v\|^2_{L^2_\rho} - 2 \int_\R\rho v v''\, dx -   \int_\R  \rho'' v v'' \, dx + \int_\R  \rho'' (v')^2 \, dx -  \| v''\|^2_{L^2_\rho} \\
   &\leq  -   \| v\|^2_{L^2_\rho} - \|v''\|^2_{L^2_\rho}
   + (2+C_2) \| v\|_{L^2_\rho} \|v''\|_{L^2_\rho}    + C_2  \| v'\|^2_{L^2_\rho} .
  \end{aligned}
   \end{multline*}

Now we use the following interpolation inequality
\begin{equation}\label{e:interpol}
\begin{split}
\| v'\|^2_{L^2_\rho}
 & =  -  \int_\R \rho' vv' \, dx -  \int_\R \rho vv'' \, dx
= \frac12 \int_\R \rho'' v^2 \, dx -  \int_\R \rho vv'' \, dx\\
  & \leq  \frac{C_2}{2} \| v\|_{L^2_\rho}^2 + \| v\|_{L^2_\rho}\| v''\|_{L^2_\rho} \;,
  \end{split}
\end{equation}
and we obtain
\begin{equation*}
	\begin{split}
		\int_\R\rho v [-(1+\partial_x^2)^2] v\, dx
			&\leq -  (1- \frac{C_2^2}{2}) \| v\|^2_{L^2_\rho} - \|v''\|^2_{L^2_\rho}
			+ 2 (1 +  C_2) \| v\|_{L^2_\rho}\|v''\|_{L^2_\rho} \\
			&\leq    -C_2 \| v''\|_{L^2_\rho}^2  + \mathcal{O}(C_2) \| v\|_{L^2_\rho}^2  \;,
	\end{split}
\end{equation*}
with the last step provided by Young's inequality. Here $\mathcal{O}(C_2) $ is just the abbreviation for a term bounded by $C\cdot C_2$, with a universal positive constant $C$.
Note that we can choose $C_2$ as small as we want by fixing $c>0$ in the definition of the weight.
\end{proof}

For the nonlinearity,
the following estimate is straightforward
\begin{lemma}
\label{lem:cubic}
For all $\frac14 >\delta>0$ we have $C_\delta = 3 - \frac9{4(1-\delta)} \in(0,3)$ such that
 \[
\langle -(v+z)^3 + z^3, v  \rangle_{L^2_\rho} \leq - C_\delta z^2 \|v\|^2_{L^2_\rho} - \delta \|v\|^4_{L^4_\rho}.
\]
\end{lemma}

\begin{proof}
By Young's inequality we have
\[
3zv^3 \leq \frac9{4(1-\delta)} z^2v^2 + (1-\delta) v^4.
\]
This implies
\begin{equation*}
 \begin{split}
 \langle -(v+z)^3 + z^3, v  \rangle_{L^2_\rho}
 & =  -  \int_\R \rho (3z^2v^2 +3zv^3+v^4 )\; dx \\
  & =  -  3z^2\int_\R \rho v^2 \;dx  - \int_\R 3z \rho v^3\; dx - \int_\R \rho v^4 \; dx \\
   & =  - C_\delta  z^2\int_\R \rho v^2 \;dx   - \delta \int_\R \rho v^4 \; dx,
 \end{split}
\end{equation*}
where $C_\delta = 3 - \frac9{4(1-\delta)} >0$.
\end{proof}


\section{Setting}
\label{sec:Set}


We focus for simplicity of presentation
on the one-dimensional stochastic Swift-Hohenberg equation although the results stated here
apply to a much more general setting.

Let us consider the stochastic equation
\begin{equation}
\label{e:SPDE}
 du= [Au+f(u) ]dt + \sigma d\beta
\end{equation}
on $\mathbb{R^d}$ where
\begin{itemize}
\item
$A$ is a nice differential operator given by $-(1+\partial_x^2)^2+\nu$ for some $\nu\in\R$ measuring the distance from bifurcation. Note that $A$ is not monotone.
Denote by $\lambda=A1$ the action of $A$ on the constant. Here $\lambda=-1+\nu$.
\item $f(u) = -u^3$ is a  polynomial nonlinearity
\item $\beta$ is a real-valued Brownian motion and $\sigma\geq0$ the noise strength.
\end{itemize}

\begin{remark}
\label{rem:gene} 
Let us comment on some generalizations. See also Section~\ref{suse:gen} and the discussion before Theorem \ref{thm:genstab}.
\begin{itemize}
\item There is nothing special about the choice of $A$. 
We could use any parabolic differential operator of even order close to bifurcation, later in the examples we also study $A=-\partial_x^4$.
If it is of second order, we could rely on monotonicity. 
\item For $f:\mathbb{R}\to \mathbb{R}$ we can use any differentiable function such that for $|u|$ sufficiently large we have $f(u)\sim -|u|^p u $ and $f'(u)\leq C$.
\item Noise not acting on the constant or higher dimensional noise is not a trivial modification,
as the stationary solution $z$ defined below is no longer a scalar valued quantity.
We need spectral information about the linear stability of the random operator $L+Df(z)$.
Later in Section \ref{sec:Stab} we also formulate a result (Theorem~\ref{thm:genstab}) for the general noise case at least in the Swift-Hohenberg setting. 
\end{itemize}
\end{remark}

We have a special case  when $\lambda = \nu$ and thus the stationary solution $z$ has a non-trivial invariant measure, which would be the case, when $A=-\partial_x^4 +\nu$.
This is significantly different to the Swift-Hohenberg equation, where $\lambda=-1+\nu$ and the invariant measure corresponding to $z$ concentrates at $z=0$. 
These will be our two examples treated later in Section~\ref{suse:gen}.


\subsection{Existence of solutions}


We assume that the following theorem holds true. It can easily be proven by a Galerkin approximation
or the approximation using bounded domains. Its proof in a slightly different setting can be found 
for example in~\cite{blohan} where a Galerkin approximation was used, 
or in \cite{biablosch16} where the problem is approximated by periodic solutions.

\begin{theorem}
 For any $u_0\in L^2_\rho$ there is a unique weak solution $u$ of \eqref{e:SPDE} in the space $L^2_{\text{loc}}([0,\infty),H^2_\rho) \cap L^\infty_{\text{loc}}([0,\infty),L^2_\rho)$.
 Moreover, the solution is spatially smooth.
\end{theorem}


\subsection{Basics on Random dynamical systems}


First in order to fix notation, let us state very briefly the well known setting of random dynamical systems. 
See the monograph of Arnold, \cite{Arn98attractors}, or the seminal paper by Crauel \& Flandoli, \cite{crafla98}, for details.

Let $(X,\|\cdot\|_X)$ be a separable real Banach space with Borel
$\sigma$-algebra $\mathcal{B}(X)$ and $(\Omega, \mathcal{F},\mathbb{P})$
be a probability space.
\begin{definition}
The quadruple $(\Omega, \mathcal{F},\mathbb{P},(\theta_t)_{t\in\mathbb{R}})$ is
called a \emph{metric dynamical system} if the map $\theta:\mathbb{R}\times
\Omega\rightarrow \Omega$ is
$(\mathcal{B}(\mathbb{R})\times\mathcal{F},\mathcal{F})$-measurable,
the map $\theta_0$ is the identity on $\Omega$,
the flow property $\theta_{s+t}=\theta_t\circ\theta_s$ holds for all $s,t\in\mathbb{R}$ and
 $\mathbb{P}$ is an invariant measure for $\theta_t$   for all $t\in\mathbb{R}$.
\end{definition}
\begin{definition}\label{DefofRDS}
A \emph{random dynamical system (RDS)} $(\theta,\phi)$   consists of a metric dynamical system $(\Omega, \mathcal{F},\mathbb{P},(\theta_t)_{t\in\mathbb{R}})$ and a co-cycle   mapping
$\phi$ $:$ $\mathbb{R}^+\times\Omega\times  X \rightarrow X$, which is  $(\mathcal{B}(\mathbb{R}^+)\times\mathcal{F}\times\mathcal{B}(X), \mathcal{B}(X))$-measurable
and satisfies the following properties:
\begin{enumerate}
\item[(i)] $\phi(0,\omega,x) = x$  \quad (initial condition)
\item[(ii)]
$\phi(s,\theta_t\omega, \phi(t,\omega,x)) = \phi(s+t,\omega,x)$
\quad (co-cycle property)
\end{enumerate}
for all $s$, $t \in \mathbb{R}^+$, $x\in X$ and $\omega \in \Omega$.
We call  a  RDS \emph{continuous} if $\phi(t,\omega,\cdot)$ is continuous with
respect to $x$ for each $t \geqslant 0$ and $\omega \in \Omega$.
\end{definition}

Denote by $2^X$ the collection of all subsets of $X$. We now define various notions of random sets (bounded, compact and invariant).
\begin{definition}
A set-valued map $B$ $:$ $\Omega \rightarrow 2^X$ is called   a \emph{random set} in $X$
if the mapping $\omega\mapsto dist(x,B(\omega))$ is $(\mathcal{F},\mathcal{B}(\mathbb{R}))$  measurable for all $x \in X$.

A \emph{random set} $B$ in $X$  is called   a \emph{random closed set} if $B(\omega)$ is non-empty and closed for each
$\omega \in \Omega$.  A random closed set is called a \emph{random compact set} if $B(\omega)$ is additionally compact for all $\omega \in \Omega$.

A random set is called \emph{$\phi$-invariant} if  for  $\mathbb{P}$-a.e.\  $\omega\in\Omega$ we have
$\phi(t,\omega,A(\omega)) = A(\theta_t\omega)$ for all $t\geqslant0$\;.
\end{definition}
Tempered sets will be important in the following,
as the pull-back attraction is exponential in many cases so we can allow for slow growth.
\begin{definition}
A random set $B$: $\Omega \rightarrow 2^X$ is called a \emph{bounded random set} if there is a non-negative random variable $R $, such that
\begin{equation*}
d(B(\omega)):=sup\{\|x\|_X:x\in B(\omega)\}\leqslant R(\omega)\quad \text{for all}~\omega\in\Omega\;.
\end{equation*}
A bounded random set is said to be  \emph{tempered} with respect to $(\Omega, \mathcal{F},\mathbb{P},(\theta_t)_{t\in\mathbb{R}})$ if for $\mathbb{P}$-a.e.\ $\omega \in \Omega$,
\begin{equation*}
\lim_{t\rightarrow\infty}e^{-\mu t} d(B(\theta_{-t}\omega))=0\quad\text{for all
}~\mu>0.
\end{equation*}
\end{definition}
In the following definition of an absorbing random set, we always think of~$\mathcal{D}$ as a family of either deterministic sets or tempered random sets.
\begin{definition}
Let $\mathcal{D}$ be collection of random sets in $X$. Then a random set $B \in \mathcal{D}$
 is called a  $\mathcal{D}$-\emph{random absorbing set} for an RDS $(\theta,\phi)$ if  for any random set
$D \in \mathcal{D}$ and $\mathbb{P}$-a.e.\
$\omega \in \Omega$, there exists a $T_{D}(\omega) > 0$ such that
\begin{equation*}
\phi(t,\theta_{-t}\omega,D(\theta_{-t}\omega)) \subset
B(\omega)\quad\mbox{for all}~t\geqslant  T_{D}(\omega).
\end{equation*}
\end{definition}
Crucial for the existence of random attractors is the notion of asymptotic compactness.
\begin{definition}
Let $\mathcal{D}$ be collection of random sets in $X$.
Then $\phi$ is said to be \emph{$\mathcal{D}$-pull-back asymptotically compact} in $X$
if  for any given $B \in \mathcal{D}$  the sequence $\{\phi(t_n, \theta_{-t_n}\omega,x_n)\}_{n=1}^\infty$ has a convergence subsequence in $X$ for $\mathbb{P}$-a.e.\  $\omega \in \Omega$,
whenever $t_n\rightarrow \infty$ and $x_n\in B(\theta_{-t_n}\omega)$.
\end{definition}
\begin{definition}\label{attractor} 
Let $\mathcal{D}$ be collection of random sets in $X$. Then a random set $A\in \mathcal{D}$ in $X$ is  called a $\mathcal{D}$-\emph{random attractor} (or $\mathcal{D}$\emph{-pull-back random attractor})
for an RDS $(\theta,\phi)$ if
\begin{enumerate}
\item[(i)] $A$ is a compact random set,
\item[(ii)] $A$ is $\phi$-invariant,
\item[(iii)] $A$ attracts every random set $D\in\mathcal{D}$, that is, for every $D \in \mathcal{D}$,
\begin{equation*}
\lim_{t\rightarrow\infty} {\rm dist_X}\left(\phi(t,\theta_{-t}\omega,
D(\theta_{-t}\omega)),A(\omega)\right)=0,\quad \mathbb{P}\text{-almost surely},
\end{equation*}
with $\text{dist}_X(Y,Z)=\sup_{y\in Y}\inf_{z\in Z}\|y-z\|_X$ for any $Y\subseteq X$ and $Z\subseteq X$  being the Hausdorff semi-metric.
\end{enumerate}
\end{definition}

For the abstract result on existence and uniqueness of random attractors we also need the universe of tempered sets.
\begin{definition}
A collection $\mathcal{D}$ of random sets in $X$ is called \emph{inclusion-closed}
if whenever $E$ is an arbitrary random set, and $F$ is in $\mathcal{D}$ with $E(\omega)\subset F(\omega)$ for all $\omega\in\Omega$, then $E$ must belong to $\mathcal{D}$.
A collection $\mathcal{D}$ of random sets in $X$ is said to be a \emph{universe} if it is inclusion-closed.
\end{definition}
The following result on the existence of pull-back attractor is now well known. See for example Bates, Lu \& Wang \cite{PhysicaD2014attractors,batluwan13}.
\begin{theorem}\label{t2.1}
Let $\mathcal{D}$ be a universe in $X$
and $(\theta,\phi)$ be a continuous RDS on~$X$.
Suppose that there exists a closed random absorbing set
$B\in\mathcal{D}$ and that $\phi$ is $\mathcal{D}$-pull-back
asymptotically compact in $X$. Then, $\phi$ has a unique
$\mathcal{D}$-random attractor $A\in\mathcal{D}$, which is given by
\begin{equation*}
A(\omega) =\bigcap_{s\geq  0}
\overline{\bigcup_{t\geq s}
\phi(t,\theta_{-t}\omega,B(\theta_{-t}\omega))}, \qquad
\omega\in\Omega.
\end{equation*}
\end{theorem}
%
\subsection{Brownian driving system}
%
Here we introduce the driving system that underlies a Brownian motion; we refer to~\cite{Arn98attractors}, for more details.

Consider the canonical probability space
$(\Omega,\mathcal{F},\mathbb{P})$ given by the Wiener space with
\[
	\Omega= \left\{\omega\in C(\mathbb{R},\mathbb{R}) : \omega(0)=0 \right\},
\]
$\mathcal{F}$ the Borel $\sigma$-algebra induced by the compact open
topology of $\Omega$, and $\mathbb{P}$ the corresponding Wiener measure on
$(\Omega,\mathcal{F})$. Thus we can identify the identity on $\Omega$ with a Brownian  motion, i.e.

\[
	\beta(t)=\omega(t)\quad\text{for}~t\in\mathbb{R}\text{ and }\omega\in\Omega.
\]
Finally, we define the time shift by
\[
	\theta_t\omega(\cdot)=\omega(\cdot+t)-\omega(t), \quad\omega\in\Omega,t\in\mathbb{R}.
\]
In that setting it is well known that $(\Omega,\mathcal{F},\mathbb{P},(\theta_t)_{t\in\mathbb{R}})$ is an ergodic metric dynamical system.

\subsection{Stationary solution}
\label{suse:SS}
We discuss now the existence of a unique spatially constant stationary solution of~\eqref{e:SPDE}.
Denote by $z$ the stationary real-valued  solution of the SDE
\begin{equation}
\label{e:SDE}
 dz= [\lambda z + f(z) ]dt + \sigma d\beta \;.
\end{equation}
By the celebrated result of Crauel \& Flandoli~\cite{crafla98}, there is only one stationary solution (up to null-sets).
The key argument is that, due to monotonicity  of the corresponding RDS,  the random attractor is a single random point.
Let us state this as a theorem.
\begin{theorem}
 There is a tempered random variable $z$ such that a stationary solution of \eqref{e:SDE} is given by $z(t,\omega)=z(\theta_t\omega)$.
 Moreover, $z$ is unique up to null-sets.
\end{theorem}
\begin{proof}
This claim is mainly from \cite{crafla98} as $z$ necessarily lies inside the unique random attractor for \eqref{e:SDE}, which is a single point.

For temperedness we just remark that the existence of a tempered absorbing set for \eqref{e:SDE} is straightforward.
Due to nonlinear stability one could even show that every set is absorbed.
As the absorbing set contains the random attractor, the random variable $z(\omega)$ is tempered.
It is given as a random ball around the stationary OU-process.
\end{proof}

To complement the previous proof, let us give a simple direct argument for $z$ to be tempered. We define  $z_2 = z- z_1$, where $z_1(\theta_t\omega)$ is the stationary OU-process solving
\begin{equation*}
dz_1+z_1dt=\sigma d\beta \;.
\end{equation*}
This is given by a tempered random variable $z_1=\sigma\int_{-\infty}^0  e^s d\beta$.
Thus it is sufficient to prove that $z_2$ is tempered.
The process  $z_1(\theta_{t}\omega)$ is also growing at most polynomially for $t\rightarrow-\infty$ and so is also
\[
 r_1(\theta_t \omega)
:=|z_1(\theta_{t}\omega)|^2+|z_1(\theta_{t}\omega)|^4 \;.
\]
On the other hand, $z_2(\theta_t\omega)$ is the stationary solution of the random ODE
\begin{equation*}
dz_2=(\nu+1)(z_1+z_2)dt-(z_1+z_2)^3 dt\;.
\end{equation*}
Differentiating $|z_2|^2$ yields, for every $\mu>0$,
\[
	\frac12\frac{d}{dt}| z_2 |^2
	\leq   C_\nu (|z_1|^2 + |z_2|^2 ) + C |z_1|^4  -  \frac12 |z_2^4
	\leq -\frac{\mu}{2}|z_2|^2 +  C_{\nu} r_1 + C_{\mu,\nu}
	\;,
\]
where we took advantage of the stable cubic, using Young's inequality.
Gr\"onwall's lemma yields
\[
	|z_2(\theta_{t}\omega)|^2
	\leq e^{-\mu t}|z_2(\omega)|^2
	+2C_{\nu}\int^{t}_0e^{-\mu(t-s)}r_1(\theta_{s}\omega) ds+C_{\mu,\nu},
\]
and thus
\[
|z_2(\omega)|^2
\leq e^{-\mu t}|z_2(\theta_{-t}\omega)|^2
+2C_{\nu}\int^{0}_{-t}e^{\mu s}r_1(\theta_{-s}\omega) ds+C_{\mu,\nu}\;,
\]
where
the following random variable is finite, as $r_1$ is tempered:
\begin{equation*}
r_2(\omega)=2\int^0_{-\infty}e^{\mu s}
(|z_1(\theta_{-s}\omega)|^2
+|z_1(\theta_{-s}\omega)|^4) ds \;.
\end{equation*}


\subsection{Transformation to a random PDE}

In order to study the random attractor for~\eqref{e:SPDE}, we transform to a random PDE using the stationary solution $z$ of the previous section.
This also shows that~\eqref{e:SPDE} generates a RDS on $L^2_\rho$.

Define
\begin{equation*}
 v(t,\omega;u_0-z(\omega)) := u(t,\omega;u_0)-z(\theta_t\omega)\;,
\end{equation*}
where $u$ is a solution of
\eqref{e:SPDE}. Then $v$ satisfies
\begin{equation}\label{e:RPDE}
\partial_t v= Av+f(v+z)-f(z),
\end{equation}
with initial data
\begin{equation*}
v(0,x)=v_0=u_0-z(\omega).
\end{equation*}

Let $u(t,\omega;u_0) := v(t,\omega;u_0-z(\omega))+z(\theta_t\omega)$. Then the stochastic process $u$
is a solution of equation \eqref{e:SPDE}. We now define a mapping
$\phi$ $:$ $\mathbb{R}^+\times\Omega\times L^2_\rho \rightarrow$
$L^2_\rho$ by
\begin{equation*}
\phi(t,\omega,u_0)=u(t,\omega;u_0) =v(t,\omega;u_0-z(\omega))+z(\theta_t\omega),
\end{equation*}
with $u(0)=u_0$. Sometimes, we also write $\phi(t,\omega,u_0)$ as
$\phi(t,\omega)u_0$. 

It follows immediately that $\phi$ satisfies the conditions
of Definition \ref{DefofRDS} and hence $(\theta,\phi)$ is a
continuous random dynamical system associated with our stochastic
parabolic equation.



\section{Random attractor for SPDE}
\label{sec:RA}


In this section we prove that the random attractor exists for~\eqref{e:SPDE}
in the weighted space $L^2_\rho$.  This is a standard proof following the ideas of Bates, Lu, \& Wang \cite{batluwan13}.
First one shows the existence of an absorbing set in $L^2_\rho$ and then additional regularity of that set for example in $H^1_\rho$.
Finally, compactness is established by bounds on the far field.

The main difficulty is the fact that we cannot rely on additional $L^p_\rho$-estimates or the maximum principle,
as in the case when $A$ is only an operator of second order.

From the rest of the section, we always assume that $\mathcal{D}$ is the universe of tempered random sets in $L^2_\rho$
with respect to the Brownian driving system  $(\Omega,\mathcal{F},\mathbb{P},(\theta_t)_{t\in\mathbb{R}})$.

\subsection{Absorbing ball for tempered sets in \texorpdfstring{$L^2_\rho$}{L2rho}}

First we show that there is a random ball with deterministic radius that pull-back attracts all tempered sets.
%
\begin{lemma}
\label{lem:L2bound}
Denote the solution of \eqref{e:SPDE} by $v(t,\omega, v_0)$  with random initial condition $v_0=u_0-z(\omega)$.
There is a constant $K=\|\rho\|_{L^1}(2\nu+C\eta_0+1)^2/4\delta$ such that for all $\omega\in\Omega$ and all $t\geq0$
\begin{equation}\label{absorbing-1}
\|v(t,\omega, v_0(\omega))\|^2_{L^2_\rho}\leq e^{-t}\|v_0(\omega)\|^2_{L^2_\rho} + K\;.
\end{equation}
Let $B=\{B(\omega)\}_{\omega\in\Omega}\in\mathcal{D}$ be a tempered set
then  for $P$-a.e.~$\omega\in\Omega$,
there is a deterministic radius $r_1=\sqrt{1+K}$
and a random time  $T_0:=T_0(B,\omega)>0$
such that for all $v_0(\omega)\in B(\omega)$ and all
$t>T_0$, we have the following estimate:
\begin{equation}\label{absorbing}
\|v(t,\theta_{-t}\omega, v_0(\theta_{-t}\omega))\|_{L^2_\rho}\leq r_1.
\end{equation}
\end{lemma}


\begin{proof}
Multiplying \eqref{e:SPDE} by  $\rho v$ we have
\begin{equation*}
\begin{split}
\frac12 \partial_t \| v \|^2_{L^2_\rho} & =-\int_{\mathbb{R}}\rho v(1+\partial_x^2)^2vdx+\nu \| v \|^2_{L^2_\rho}+\int_{\mathbb{R}}\rho v(f(v+z)-f(z))dx\\
& \leq- \eta_0 \|v''\|_{L^2_\rho}^2 + C\eta_0 \|v\|_{L^2_\rho}^2  +\nu  \| v \|^2_{L^2_\rho}- C_\delta z^2 \|v\|^2_{L^2_\rho} - \delta \|v\|^4_{L^4_\rho}
\;.
\end{split}
\end{equation*}
This implies
\begin{multline*}
\frac12 \partial_t \| v \|^2_{L^2_\rho}+\frac12 \| v \|^2_{L^2_\rho}+\eta_0 \|v''\|_{L^2_\rho}^2
+C_\delta z^2 \|v\|^2_{L^2_\rho}+\delta \|v\|^4_{L^4_\rho}\\
\leq(C\eta_0+\nu+\frac12) \|v\|_{L^2_\rho}^2 \;.
\end{multline*}
With  $C_{\nu,\eta_0,\delta}= (2\nu+C\eta_0+1)^2 / 8\delta$ we obtain
\begin{equation}\label{e:L2}
\partial_t \| v \|^2_{L^2_\rho}+\| v \|^2_{L^2_\rho}+2\eta_0 \|v''\|_{L^2_\rho}^2
+2C_\delta z^2 \|v\|^2_{L^2_\rho}+\delta \|v\|^4_{L^4_\rho}\leq 2C_{\nu,\eta_0,\delta}\|\rho\|_{L^1},
\end{equation}
and thus Gr\"onwall's lemma yields the following inequality for all $s\geq 0$,
\begin{equation}\label{4}
\begin{split}
\| v(s,\omega,v_0(\omega)) \|^2_{L^2_\rho}
&\leq e^{-s}\| v_0(\omega) \|^2_{L^2_\rho}+ 2C_{\nu,\eta_0,\delta}\|\rho\|_{L^1}\int^s_0 e^{-(s-\tau)}d\tau\\
&\leq e^{-s}\| v_0(\omega) \|^2_{L^2_\rho}+ 2C_{\nu,\eta_0,\delta}\|\rho\|_{L^1}\;.
\end{split}
\end{equation}
This implies \eqref{absorbing-1}. For \eqref{absorbing}, we replace $\omega$ by $\theta_{-t}\omega$ with $t\geq 0$ to obtain
\begin{equation*}
\| v(t,\theta_{-t}\omega,v_0(\theta_{-t}\omega)) \|^2_{L^2_\rho}\leq e^{-t}\| v_0(\theta_{-t}\omega) \|^2_{L^2_\rho}+ 2C_{\nu,\eta_0,\delta}\|\rho\|_{L^1}\;.
\end{equation*}
As $B$ is a tempered  set   and $v_0\in B$,  there exists $T_0(B,\omega)>0$ such that for all $t\geq T_0$,
\begin{equation}
\begin{split}
\| v(t,\theta_{-t}\omega,v_0(\theta_{-t}\omega)) \|^2_{L^2_\rho}&\leq e^{-t}\| v_0(\theta_{-t}\omega) \|^2_{L^2_\rho}+ 2C_{\nu,\eta_0,\delta}\|\rho\|_{L^1}\\
&\leq 1+2C_{\nu,\eta_0,\delta}\|\rho\|_{L^1}:=r^2_1 \;.
\end{split}\label{e:defr1}
\end{equation}
\end{proof}

\subsection{Additional regularity}

Here we establish additional regularity in addition to the $L^2_\rho$-boundedness in the pull-back sense,
which is necessary in the proof of existence of a compact absorbing set.

\begin{lemma}
\label{lem:addreg}
Under the assumptions of Lemma~\ref{lem:L2bound}
 the solution $v$ of~\eqref{e:RPDE} satisfies for all $\omega\in\Omega$ and all $t\geq0$
 \begin{equation*}
\int^{t+1}_{t}
\|v''(s,\omega,v_{0}(\omega))\|_{L^2_\rho}^2ds\leq \frac{1}{2\eta_0} [ e^{-t}\| v_0(\omega) \|^2_{L^2_\rho} + K ],
\end{equation*}
and
\begin{equation*}
\int^{t+1}_{t}
\|v'(s,\omega,v_{0}(\omega))\|_{L^2_\rho}^2ds\leq \frac{(2C_2\eta_0+2\eta_0+1)}{4\eta_0 } \cdot [ e^{-t}\| v_0(\omega) \|^2_{L^2_\rho} + K ].
\end{equation*}
Moreover, for all tempered sets $B$ of initial conditions we have the following estimates for all $t>T_1:=T_0+1$ :
\begin{equation*}
\int^{t+1}_{t}
\|v''(s,\theta_{-t-1}\omega,v_{0}
(\theta_{-t-1}\omega))\|_{L^2_\rho}^2ds\leq \frac{5 r^2_1}{2\eta_0},
\end{equation*}
and
\begin{equation*}
\int^{t+1}_{t}
\|v'(s,\theta_{-t-1}\omega,v_{0}
(\theta_{-t-1}\omega))\|_{L^2_\rho}^2ds\leq \frac{5(2C_2\eta_0+2\eta_0+1)}{4\eta_0 } \cdot r^2_1.
\end{equation*}
\end{lemma}


\begin{proof}

We start with~\eqref{e:L2}. Integrating the inequality from $t$ to $t+1$ and neglecting some positive terms on the left hand side yields for $t>T_0+1$
\begin{equation}
\label{e:addreg2}
\begin{split}
 \int_t^{t+1} \| v(s,\omega,v_0(\omega)) \|^2_{L^2_\rho} ds
 & +2\eta_0 \int_t^{t+1} \|v''(s,\omega,v_0(\omega))\|_{L^2_\rho}^2 ds \\
&\leq  \| v(t,\omega,v_0(\omega)) \|^2_{L^2_\rho} + 2C_{\nu,\eta_0,\delta}\|\rho\|_{L^1}\\
&\leq  e^{-t}\| v_0(\omega) \|^2_{L^2_\rho} + K  
\;,
\end{split}
\end{equation}
where we used inequality~\eqref{absorbing-1}. This gives the first bound.

In order to bound the $L^2_\rho$-norm, we consider once again~\eqref{4}.
Replacing $\omega$ by $\theta_{-t-1}\omega$ and $s$ by $t$ we derive  the following inequality:
\begin{equation}
\begin{split}
\| v(t,\theta_{-t-1}\omega,v_0(\theta_{-t-1}\omega)) \|^2_{L^2_\rho}
& \leq e^{-t} \| v_0(\theta_{-t-1}\omega) \|^2_{L^2_\rho}+ 2C_{\nu,\eta_0,\delta}\|\rho\|_{L^1}  \\
&\leq e \cdot e^{-t-1} \| v_0(\theta_{-t-1}\omega) \|^2_{L^2_\rho}+ 2C_{\nu,\eta_0,\delta}\|\rho\|_{L^1}  \\
&\leq 4 r^2_1,
\end{split}
\label{e:addreg1}
\end{equation}
for $t>T_0+1$,
where we used \eqref{absorbing} from Lemma \ref{lem:L2bound} and the definition of $r_1$.

Substituting in \eqref{e:addreg2} $\omega$ with $\theta_{-t-1}\omega$ yields
\begin{equation*}
\begin{split}
\int_t^{t+1} \!\! \| v(s,\theta_{t-1}\omega,v_0(\theta_{-t-1}\omega)) \|^2_{L^2_\rho} ds
& +2\eta_0 \int_t^{t+1} \!\! \|v''(s,\theta_{-t-1}\omega,v_0(\theta_{-t-1}\omega))\|_{L^2_\rho}^2 ds\\
&\leq  \| v(t,\omega,v_0(\theta_{-t-1}\omega)) \|^2_{L^2_\rho} + 2C_{\nu,\eta_0,\delta}\|\rho\|_{L^1}\\
&\leq  5r_1^2
\;,
\end{split}
\end{equation*}
where we used the bounds~\eqref{e:addreg1} and~\eqref{e:defr1}. This implies the first claim.

For the bound on the $H^1_\rho$-norm we use that, by interpolation~\eqref{e:interpol} and Young's inequality,
\[
\|v'\|^2_{L^2_\rho} \leq \frac{C_2+1}2  \|v\|^2_{L^2_\rho} + \frac12 \|v''\|^2_{L^2_\rho}
\leq \max\{\frac{C_2+1}2,  \frac1{4\eta_0}  \}(  \|v\|^2_{L^2_\rho}+  2\eta_0 \|v''\|^2_{L^2_\rho} ) \;.
\]
Thus using~\eqref{e:addreg2} we can conclude
\[
\int^{t+1}_{t}
\|v'(s,\theta_{-t-1}\omega,v_{0}
(\theta_{-t-1}\omega))\|_{L^2_\rho}^2ds
\leq  5r_1^2 \max\{\frac{C_2+1}2,  \frac1{4\eta_0}  \}\;.\qedhere
\]
\end{proof}

\subsection{Absorbing set bounded in \texorpdfstring{$H^1_\rho$}{H1rho}}

We prove now that there is an $L^2_\rho$-absorbing set, which is bounded in $H^1_\rho$.

\begin{lemma}
There is a constant $C$   such that for all $\omega\in\Omega$ and all $t\geq 0$
\begin{equation*} 
\|v(t+1,\omega, v_0(\omega))\|^2_{H^1_\rho}\leq C[e^{-t}\|v_0(\omega)\|^2 +1]\;.
\end{equation*}
Moreover, there exists  a deterministic radius $r_2$ such that
for every tempered set $B=\{B(\omega):\omega\in\Omega\}\in\mathcal{D}$ of initial conditions $v_0(\omega)\in B(\omega)$, we obtain that for all $t>T_1$, with $T_1$ the random time  from Lemma~\ref{lem:addreg},

\begin{equation*} 
\|v(t,\theta_{-t}\omega, v_0(\theta_{-t}\omega))\|_{H^1_\rho}\leq r_2\;.
\end{equation*}
\end{lemma}


\begin{proof}
It is enough to consider the norm  $\|v'\|^2_{L^2_\rho}$. Differentiating this and using~\eqref{e:RPDE} yields
\begin{equation*}
\begin{split}
\frac{1}{2}\frac{d}{dt}\|v'\|^2_{L^2_\rho}
& = \langle v' , \partial_t v' \rangle_{L^2_\rho}\\
& = -\langle(1+\partial_x^2)^2v',\rho v' \rangle_{L^2}+\nu\|v'\|_{L^2_\rho}
+\langle (f(v+z)-f(z))',\rho v'\rangle_{L^2}\;. 
\end{split}
\end{equation*}
Now Lemma~\ref{lem:spec} implies
\begin{equation*}
- \langle (1+\partial_x^2)^2v',\rho v' \rangle_{L^2}
\leq -\eta_0\|v'''\|^2_{L^2_\rho}+C\eta_0\|v'\|^2_{L^2_\rho}.
\end{equation*}
For the nonlinear term,
\begin{equation*}
\begin{split}
\langle(f(v+z)-f(z))',\rho v'\rangle_{L^2}
&=-\int_\R \rho (3z^2v +3zv^2+v^3 )'v'\; dx
\\  & =  -  3z^2\int_\R \rho v'^2 \;dx  - \int_\R 6z \rho vv'^2\; dx - 3\int_\R \rho v^2v'^2 \; dx \\
   & \leq 0 \;.
\end{split}
\end{equation*}
Then we have for every $\omega\in\Omega$,
\begin{equation*}
\frac{1}{2}\frac{d}{dt}\|v'\|^2_{L^2_\rho}
+\eta_0\|v'''\|^2_{L^2_\rho}\leq (C\eta_0+\nu)\|v'\|^2_{L^2_\rho}\;.
\end{equation*}
Integrating the inequality first w.r.t.\ time over $(s,t+1)$ yields
\begin{multline*}
\|v'(t+1,\omega,v_0(\omega))\|^2_{L^2_\rho}\leq \|v'(s,\omega,v_0(\omega))\|^2_{L^2_\rho}
\\ +(C\eta_0+\nu)
\int^{t+1}_{s}
\|v'(\tau,\omega,v_0(\omega))\|^2_{L^2_\rho}d\tau,
\end{multline*}
and integrating now w.r.t.\ $s$ over $[t,t+1]$ we obtain
\begin{equation*}
\|v'(t+1,\omega,v_0(\omega))\|^2_{L^2_\rho}\leq (C\eta_0+\nu+1)
\int^{t+1}_t
\|v'(\tau,\omega,v_0(\omega))\|^2_{L^2_\rho}dr\;.
\end{equation*}
Now Lemma~\ref{lem:addreg} yields the first claim.

For the second claim,
replacing $\omega$ by $\theta_{-t-1}\omega$, we have for $t>T_1$, the random time from Lemma~\ref{lem:addreg}:
\begin{multline*}
\|v'(t+1,\theta_{-t-1}\omega,v_0(\theta_{-t-1}\omega))\|^2_{L^2_\rho}\\
\begin{aligned}
&\leq (C\eta_0+\nu+1)
\int^{t+1}_{t}
\|v'(s,\theta_{-t-1}\omega,v_0(\theta_{-t-1}\omega))\|^2_{L^2_\rho}dr\\
&\leq \frac{5(2C\eta_0+2\nu+1)(C_2\eta_0+\eta_0+1)}{4\eta_0 }r^2_1:=r^2_2. \qedhere
\end{aligned}
\end{multline*}
\end{proof}


\subsection{Uniform integrability}


In order to proof the compactness of our absorbing set, we prove first uniform integrability. Therefore, 
we use a cut-off function.
Let $\varphi(\cdot)\in\mathcal{C}^\infty(\mathbb{R})$ be symmetric such that $0\leq \varphi(s)\leq 1$ for all $s\in\mathbb{R}$ and
\begin{equation*}
\varphi(s)=
\begin{cases}
0 \qquad \text{for } |s| \leq 1\;,
\\[1.5ex]
1 \qquad \text{for } |s| \geq 2\;.
\end{cases}
\end{equation*}
Now let us define $\varphi_r(x)=\varphi(x/r)$ for $r\geq1$.

\begin{lemma}
\label{lem:uniint}
Let $v$ be the solution of~\eqref{e:RPDE}. Then
for any $\varepsilon>0$, there exist a deterministic $R_\varepsilon\geq1$, a deterministic time $T_\varepsilon\geq1$
and some constant $C>0$ such that for $t\geq T_\varepsilon$ and $r\geq R_\varepsilon$ and all $\omega\in\Omega$
\begin{equation*}
\|\varphi_rv(t,\omega,v_0(\omega))\|^2_{L^2_\rho}
\leq C (t+1) e^{-t} \|v_0(\omega)\|^2_{L^2_\rho}    + \frac12 \varepsilon\;.
\end{equation*}
Moreover, for any tempered set $B=\{B(\omega):\omega\in\Omega\}\in\mathcal{D}$
and solutions starting in $v_0(\omega)\in B(\omega)$,
 there exists a random time $T_3$ depending also on $B$ and $\varepsilon$ such that for all $t>T_3$ and $r>R_\varepsilon$
\begin{equation*}
\int_{|x|>r}\rho
|v(t,\theta_{-t}
\omega,v_0(\theta_{-t}
\omega))(x)|^2\; dx \leq \varepsilon.
\end{equation*}
\end{lemma}
%
%
\begin{proof}
We consider the $L^2_{\rho\varphi_r^2}$-norm of $v$, we differentiate it and by~\eqref{e:RPDE} we get
\begin{equation}
\label{e:cutL2}
\frac{1}{2}\frac{d}{dt}\|\varphi_r  v\|^2_{L^2_\rho}
= \langle -(1+\partial_x^2)^2v, \varphi_r^2 v\rangle_{L^2_\rho}+\nu\|\varphi_rv\|^2_{L^2_\rho}+\langle f(v+z)-f(z),\varphi_r^2 v\rangle_{L^2_\rho }\;.
\end{equation}
We now bound each term on the right hand side separately. First we obtain for the quadratic form of the differential operator
\begin{equation*}
\langle -(1+\partial_x^2)^2v, \varphi^2_r v\rangle_{L^2_\rho}
=\langle -(1+\partial_x^2)^2(\varphi_rv),\varphi_r v\rangle_{L^2_\rho}
+I_1  \;,
\end{equation*}
where
\begin{multline*}
I_1 =
\langle v(2\partial_x^2+\partial_x^4)\varphi_r, \varphi_r v\rangle_{L^2_\rho} \\
+\int_{\R}( 2 \rho\varphi_r v v' \partial_x\varphi_r
 +6\rho\varphi_rvv''\partial_x^2\varphi_r
+4\rho\varphi_rvv'\partial_x^3\varphi_r
+4\rho\varphi_rvv'''\partial_x\varphi_r)dx\;.
\end{multline*}
From Lemma \ref{lem:spec} we obtain
\begin{equation*}
\langle -(1+\partial_x^2)^2(\varphi_rv),\varphi_r v\rangle_{L^2_\rho}
\leq -\eta_0\|(\varphi_rv)''\|^2_{L^2_\rho}+C\eta_0\|\varphi_rv\|^2_{L^2_\rho} \;.
\end{equation*}
It now remains to bound all terms in $I_1$, where all contributions will be small in $r$. First,
\begin{equation*}
|\langle v(2\partial_x^2+\partial_x^4)\varphi_r, \varphi_r v\rangle_{L^2_\rho}|
\leq \frac{C}{r}\|v\|^2_{L^2_\rho},
\end{equation*}
where we used that all derivatives of $\varphi_r$ are bounded uniformly in $x$ by $\mathcal{O}(r^{-1})$, as $r\geq1$.
For the remaining terms in $I_1$, we can argue similarly,
using integration by parts and the fact that derivatives of $\rho$ are bounded by $\mathcal{O}(\rho)$.
After some calculations we obtain:
\[
\Bigg|\int_{\R}(2\rho\varphi_r\partial_x\varphi_r
+4\rho\varphi_r\partial_x^3\varphi_r)\underbrace{vv'}_{=\frac12(v^2)'}dx\Bigg|
\leq \frac{C}{r}\|v\|_{L^2_\rho}^2\;.
\]
Moreover,
\[
\Bigg|\int_{\R}\rho\varphi_r\partial_x^2\varphi_r \cdot vv''\; dx\Bigg|
\leq \frac{C}{r}[\|v'\|^2_{L^2_\rho} + \|v\|^2_{L^2_\rho} ],
\]
and using that $2v'v'' = ((v')^2)'$ finally we have
\[
\Bigg|\int_{\R}\rho\varphi_r  \partial_x\varphi_r\cdot vv'''\; dx\Bigg|
\leq \frac{C}{r}[\|v'\|^2_{L^2_\rho} + \|v\|^2_{L^2_\rho} ]\;.
\]
Thus the final result for the quadratic form is
\begin{equation*}
 \langle -(1+\partial_x^2)^2v, \varphi^2_r v\rangle_{L^2_\rho}
\leq -\eta_0\|(\varphi_rv)''\|^2_{L^2_\rho}+C\eta_0\|\varphi_rv\|^2_{L^2_\rho}
+ \frac{C}{r} \|v'\|^2_{H^1_\rho} \;.
\end{equation*}
For the nonlinear term we obtain from Lemma \ref{lem:cubic} with $C_\delta = 3 - \frac9{4(1-\delta)} >0$
\begin{equation*}
\begin{split}
\langle f(v+z)-f(z), \varphi^2_r v\rangle_{L^2_\rho} &= \langle f(v+z)-f(z), v\rangle_{L^2_{\rho \varphi^2_r}}\\
 & \leq  - C_\delta  z^2\int_\R \varphi^2_r\rho v^2 \;dx   - \delta \int_\R \varphi^2_r\rho v^4 \; dx \;. 
\end{split}
\end{equation*}
Combining the previous two inequalities with \eqref{e:cutL2} yields
\begin{multline*}
\frac{1}{2}\frac{d}{dt}\|\varphi_rv\|^2_{L^2_\rho}
+\frac{1}{2}\|\varphi_rv\|^2_{L^2_\rho}
+\eta_0\|(\varphi_rv)''\|^2_{L^2_\rho}+C_\delta  z^2\int_\R \varphi^2_r\rho v^2 \;dx+\delta \int_\R \varphi^2_r\rho v^4\\
\leq (C\eta_0+\frac12)\|\varphi_rv\|^2_{L^2_\rho}
+\frac{C}{r}\|v\|^2_{H^1_\rho}\;.
\end{multline*}
Using Young's inequality, we get rid of the $L^2_\rho$-norm on the right hand side:
\begin{multline*}
\frac{1}{2}\frac{d}{dt}\|\varphi_rv\|^2_{L^2_\rho}
+\frac{1}{2}\|\varphi_rv\|^2_{L^2_\rho}
+\eta_0\|(\varphi_rv)''\|^2_{L^2_\rho}+C_\delta  z^2\int_\R \varphi^2_r\rho v^2 \;dx+\frac{\delta}{2} \int_\R \varphi^2_r\rho v^4\\
\leq 2K_0\|\varphi^2_r\rho\|_{L^1}
+\frac{C}{r}\|v\|^2_{H^1_\rho}\;,
\end{multline*}
where we defined $K_0=(2C\eta_0+1)^2 / 4\delta$.
We finally conclude
\begin{equation*}
\frac{d}{dt}\|\varphi_rv\|^2_{L^2_\rho}
+\|\varphi_rv\|^2_{L^2_\rho}
\leq K_0 \|\varphi^2_r\rho\|_{L^1}
+ \frac{C}{r} \|v\|^2_{H^1_\rho}.
\end{equation*}
We now use Lemma \ref{lem:L2bound} and Lemma \ref{lem:addreg} to proceed.
Applying comparison principle for ODEs  we obtain  for $t\geq 0$   (the constant $C$ might be different in different places)
\begin{multline*}
\|\varphi_rv(t+1,\omega,v_0(\omega))\|^2_{L^2_\rho} \leq e^{-t}\|\varphi_r v(1,\omega,v_0(\omega))\|^2_{L^2_\rho} + K_0 \|\varphi^2_r\rho\|_{L^1}\\
\qquad\qquad\qquad+ \frac{C}{r} \int_0^t e^{-(t-s)} \|v(s+1,\omega,v_0(\omega))\|^2_{H^1_\rho} ds \\
\leq
e^{-t}  [e^{-1}\|v_0(\omega)\|^2_{L^2_\rho}+K]  +  K_0 \|\varphi^2_r\rho\|_{L^1}
+ \frac{C}{r} [ t e^{-t}\|v_0(\omega)\|^2 +1].
\end{multline*}
Note furthermore that
\begin{equation*}
\|\varphi^2_r\rho\|_{L^1}
\leq\int_{|x|>r}\rho(x)dx\rightarrow 0,\quad\text{as}~r\rightarrow\infty\;,
\end{equation*}
thus  for any $\varepsilon>0$, there exist a deterministic $R_\varepsilon\geq1$ and a deterministic time $T_\varepsilon\geq1$ such that for $t\geq T_\varepsilon$ and $r\geq R_\varepsilon$ we have
\begin{equation*}
\|\varphi_rv(t,\omega,v_0(\omega))\|^2_{L^2_\rho}
\leq C (t+1) e^{-t} \|v_0(\omega)\|^2_{L^2_\rho}    + \frac12 \varepsilon,
\end{equation*}
for some constant $C>0$.

For the final claim, replacing $\omega$ by $\theta_{-t}\omega$ and noting that $v_0$ is in a tempered set,
we obtain the existence of a random time $T_3$ (depending on $\varepsilon$) such that
\begin{equation*}
\|\varphi_rv(t,\theta_{-t}\omega,v_0(\theta_{-t}\omega))\|^2_{L^2_\rho} \leq \varepsilon
\quad\text{for all }t \geq T_3.\qedhere
\end{equation*}
\end{proof}


\subsection{Random atractor}

We could now prove the existence for a random attractor for~\eqref{e:RPDE} both forward and backward in time, but we directly aim for the random attractor for~\eqref{e:SPDE}.
Note that we have the following transformation, such that for each $t\geq 0$ and $\omega\in\Omega$
\begin{equation}\label{3}
u(t,\omega,u_0)
=v(t,\omega,u_0-z(\omega))
+z(\theta_t\omega) \;.
\end{equation}
Suppose $B=\{B(\omega):\omega\in\Omega\}$ is a tempered
family of non-empty subsets of $L^2_\rho$.
Given $B$ we define the family $\tilde{B}$ by
\begin{equation*}
\tilde{B}(\omega)=\{w\in L^2_\rho:\|w\|^2_{L^2_\rho}
\leq 2\cdot \text{diam}_{L^2_\rho}(B(\omega))^2
+2|z(\omega)|^2\}\;.
\end{equation*}
The following properties hold:
\begin{itemize}
\item if $B$ is tempered, then also $\tilde{B}$ is tempered (because $z(\omega)$ is tempered);
\item if $u(\theta_{-t}\omega)\in B(\theta_{-t}\omega)$, then  $v(\theta_{-t}\omega)=u(\theta_{-t}\omega)-z(\theta_{-t}\omega) \in\tilde{B}(\theta_{-t}\omega)$.
\end{itemize}
Hence, we obtain estimates for $u$ immediately from~\eqref{3} and the corresponding bounds for $v$.

\begin{lemma}
\label{lem:uabsorb}
There exists a deterministic radius $r>0$ such that
for all tempered $B=\{B(\omega):\omega\in\Omega\}$ and $u_0\in B(\omega)$ there is a a random time  $T$ such that for all $t>T$
the solution of~\eqref{e:SPDE} satisfies
\begin{equation*}
\|u(t,\theta_{-t}\omega, u_0(\theta_{-t}\omega))\|_{H^1_\rho}\leq r_2+|z(\omega)|\;.
\end{equation*}
\end{lemma}

\begin{lemma}
\label{lem:ucomp}
For any $\varepsilon>0$ there is a random $R=R(\omega,\varepsilon)\geq 1$ such that
for any tempered sets $B=\{B(\omega):\omega\in\Omega\}$ and initial conditions $u_0\in B(\omega)$
there exists a random time $T=T(\omega,\varepsilon,B)$
such that for all $t>T$ and $r>R$, the solution of~\eqref{e:SPDE} satisfies
\begin{equation*}
\int_{|x|>r}\rho
|u(t,\theta_{-t}\omega, u_0(\theta_{-t}\omega))(x)|^2dx \leq \varepsilon.
\end{equation*}
\end{lemma}
In contrast to Lemma~\ref{lem:uniint}, the radius $R$ is now random, as we need that $|z(\omega)|^2 \int_{|x|>r} \rho dx \leq \varepsilon$.


We can now formulate the compactness result for the RDS given by~\eqref{e:SPDE}.
\begin{theorem}
\label{thm:comp}
The RDS $(\theta,\phi)$ generated by~\eqref{e:SPDE} is asymptotically compact in $L^2_\rho$.
This means that for every $\omega\in\Omega$, every tempered set $B=\{B(\omega):\omega\in\Omega\}$, sequence $t_n\rightarrow\infty$,
and initial conditions $u_{n}(\omega)\in B(\omega)$,
the sequence $u(t_n,\theta_{-t_n}\omega,u_{n}(\theta_{-t_n}\omega))$ has a convergent subsequence in $L^2_\rho$.
\end{theorem}


\begin{proof}
Fixed $\omega\in\Omega$
we need to show that  the sequence
\[
U_n(\omega)=u(t_n,\theta_{-t_n}\omega,u_{0,n}(\theta_{-t_n}\omega))
\]
is relatively compact (or totally bounded), i.e. for every $\varepsilon>0$ it
has a finite covering of balls of radii less than $\varepsilon$.

By Lemma~\ref{lem:ucomp} there exists $R$ and $N$ such that for all $n\geq N$,
\begin{equation*}
\|U_n\|_{L^2_\rho(\R\setminus Q_R)}
\leq {\varepsilon}
\qquad\text{where }Q_R=\{x\in\R:|x|\leq R\}.
\end{equation*}

On the other hand, by Lemma \ref{lem:uabsorb} there exists $N_\star\geq N$ such that for all $n\geq N_\star$,
\begin{equation*}
\|U_n\|_{H^1_\rho(Q_R)}
\leq r + |z(\omega)|\;.
\end{equation*}
Using the compact embedding of  $H^1_\rho(Q_R)$ into $L^2_\rho(Q_R)$, the sequence $U_n$ is precompact in $L^2_\rho(Q_R)$,
and thus has a finite covering in $L^2_\rho(Q_R)$ of balls of radii less than ${\varepsilon}$.

Combining both parts, $U_n$ has a finite covering of ball of radii less than $2\varepsilon$ in $L^2_\rho(\R)$.
\end{proof}


From  Theorem~\ref{thm:comp} and Lemma \ref{lem:uabsorb}, we can use Theorem~\ref{t2.1}
 to obtain the existence of the pull-back attractor.
\begin{theorem}
\label{thm:exRDS}
The RDS $(\theta,\phi)$ generated by~\eqref{e:SPDE} has a unique pull-back attractor $\mathcal{A}$ in $L^2_\rho$.
\end{theorem}


\subsection{Generalizations}
\label{suse:gen}

The result on existence of random attractor stated in Theorem~\ref{thm:exRDS}
in the setting of the Swift-Hohenberg equation holds in a much more general setting.
We state a few remarks about the modifications necessary in the arguments for such generalizations.

\begin{remark}[Linear Operator]
For the linear operator $A=p(-\Delta)$ we could consider other functions $p$ of the Laplacian $\Delta$,
and even odd derivatives of lower order.
The main estimate we need to establish is the key result of Lemma~\ref{lem:spec}. But we do not need non-negativity of the operator.
We mainly need for some $k>d/2$, where $d$ is the dimension of the underlying domain, that there are constants $c,C>0$ such that 
\[
\langle Av,v\rangle_{L^2_\rho} \leq - c\|v\|^2_{H^k_\rho} + C \|v\|^2_{L^2_\rho}\;.
\]
This should hold in the case of the function $p$ being bounded from above on $[0,\infty)$ with sufficiently fast growth at infinity,
like an even polynomial with negative leading order coefficients.
\end{remark}

\begin{remark}[Nonlinearity]
We mainly treat the stable cubic $-u^3$ in all our examples. 
It is essential that the nonlinearity induces a stronger non-linear stability in $L^p$-spaces,
but we are not restricted to the cubic. It should be a straightforward generalization to use the whole machinery 
for general nonlinearities $f:\R\to\R$, where $f(u) \sim -u|u|^{p-2}$ for some $p>2$ and large $u\to \pm \infty$.  
So, for example, any polynomial of odd degree with negative leading coefficient should be possible. 
\end{remark}

\begin{remark}[Higher dimension]
All the results presented so far treat the one-dimensional case, but the results do generalize to higher dimensions.
In the case of $x\in \mathbb{R}^2$ for the Swift-Hohenberg setting of \eqref{e:SPDE}, 
we could also obtain the existence of random attractor in $L^2_\rho$, as most estimates for the linear operator and the nonlinearity do not rely on the dimension. 

Let us remark that the existence result needs a Sobolev embedding of the $H^k$-space, controlled by the linear operator,
into $C^0$-spaces at least locally on bounded domains.  So here we might need a restriction to lower dimensions.
\end{remark}

\begin{remark}[General Noise]
It is non trivial to consider general additive noise, as in some estimates we explicitly rely on the scalar nature of the stationary solution $z$.
Moreover, in our case we can rely on well-known SDE results for $z$, which are not easily established in the case of general additive noise given for instance by the derivative of a $Q$-Wiener process.
It is not straightforward to establish results, Especially, when the noise is translation-invariant (i.e., spatially stationary); 
even the existence of stationary solutions does not seem to be established yet, although it should be possible by  adapting standard results.

For the existence of RDS and random attractors in the general noise case, one would consider $z$ only as the stationary solution of the stationary stochastic convolution,
i.e.\ the Ornstein-Uhlenbeck process solving the linearised SPDE. But this would require changes to many of the estimates presented here. 
\end{remark}


\section{Size of attractor for deterministic PDE}
\label{sec:SoA}


In this section, we 
consider $\sigma=0$ and state some results and conjectures on the size  of the deterministic attractor for $\nu>0$.
We provide straightforward bounds on the diameter of the attractor in $L^2_\rho$ and discuss
lower bounds on the dimension, which should be high dimensional for $\nu>0$ as an infinite band of eigenvalues changes stability.

For the discussion we restrict ourselves to the case 
\begin{equation}
\label{e:PDE}
 \partial_tu = A u + \nu u - u^3,
\end{equation}
posed on $\mathbb{R}$ where $A=-(1+\partial_x^2)^2$ or $A=-\partial_x^4$.  Let us first remark that the existence of a deterministic attractor 
in the space $L^2_\rho$ follows analogously to the estimates presented in the previous sections with $z=0$, 
as we never used that the noise is non-zero. 

Let us first state, as a conjecture without proof, why we think that the attractor is infinite dimensional.

\begin{remark}[High dimensional attractor]
In the case of $\nu>0$ and no noise (i.e., $\sigma=0$),
a bifurcation analysis of the PDE~\eqref{e:PDE}
posed on $[0,L]$ with periodic boundary conditions shows,
that there is a continuous interval $I= I(\nu)\subset\R$  of periods such that for all $L\in I$ there is a
non-trivial $L$-periodic stationary solution that has no smaller periodicity.

As all these stationary solutions are in $L^2_\rho$ for an integrable weight $\rho$, the deterministic attractor
for the PDE posed in $L^2_\rho$ contains for $\nu>0$ immediately uncountably many stationary solutions.
Moreover, all translations  of such periodic solutions are again stationary solutions. 

In the Swift-Hohenberg case the set $I$ is an interval around $2\pi$, 
while for  $A=-\partial_x^4$ it is of the type $I=[L_0,\infty)$. 
\end{remark}

But the attractor does not only contain periodic solutions. Here we restrict to the second example, as for Swift-Hohenberg 
there are constant stationary solutions only if  $\nu>1$.

\begin{remark}[Spatial heteroclinic]
Consider the PDE \eqref{e:PDE} with $A=-\partial_x^4$, and suppose for the weight $\rho>1$, so that constants are in $L^2_\rho$.
It should be possible, at least for small $\nu$, to show that there is also at least one heteroclinic solution of the scalar ODE $\partial_x^4 u = \nu u - u^3$
connecting the fixed-points $\pm\sqrt{\mu}\in\R$ via the unstable fixed-point $0$.
This defines a non-periodic stationary solution of our PDE, again with the property that all translations are again stationary solutions.
So the attractor contains an infinitely long continuous curve of stationary solutions that connects the spatially constant solutions $u=\pm\sqrt{\mu}\in L^2_\rho$.
\end{remark}

Let us finally state an observation on the diameter of the random attractor in $L^2_\rho$. 
For $\nu>0$ and no noise (i.e., $\sigma=0$) in case of $A=-\partial_x^4$
it is a simple observation that
the constant solutions $\pm\sqrt{\nu}$
are in the attractor. Thus in $L^2_\rho$
the diameter of the attractor is bounded from below  by
\[2\sqrt{\nu}\|\rho\|_{L^1}^{1/2} \sim C \sqrt{\nu/c}.
\]
But we believe that for Swift-Hohenberg a similar bound is true by restricting to $2\pi$-periodic solutions. 
For any $\nu>0$ we find a  stationary $2\pi$-periodic solution of \eqref{e:PDE} with $L^2[0,2\pi]$-norm of the order $\sqrt{\nu}$. 

From~\eqref{e:L2} for Swift-Hohenberg we can get an upper bound of the type: 
\[
 \|v\|^4_{L^2_\rho} \leq C\frac1{\delta^2} (\eta_0^2+\nu^2) \|\rho\|_{L^1}^2 \sim  C\frac1{\delta^2c^2} (c^4+\nu^2)  \,.
\]
But a similar result holds true in the case of $A=-\partial_x^4$. Thus we obtain the final result,
which we state as a conjecture, as we did not prove it in detail.

\begin{remark}[Diameter of the attractor]
For $\nu>0$ and no noise (i.e., $\sigma=0$) if we consider small constants $c>0$ in the definition of the weight $\rho$, then we conjecture that
the diameter of the attractor scales like $\sqrt{\nu/c}$.
\end{remark}


\section{Stabilization}
\label{sec:Stab}


In this section we present the main results on stabilization. We will need the following assumption
\begin{assumption} \label{a:stab}
There are constants $\eta>0$ and $C_\delta>0$ such that for all $v\in L^2_\rho$ and $z\in\mathbb{R}$ 
 \[
\langle v,Av \rangle_{L^2_\rho}
\leq
\eta \|v\|^2_{L^2_\rho},
\]
and
\[
\langle f(v+z)-f(z), v  \rangle_{L^2_\rho} \leq - C_\delta z^2 \|v\|^2_{L^2_\rho}.
\]
\end{assumption}

This assumption is for example satisfied (using Lemmas \ref{lem:spec} and \ref{lem:cubic}) for Swift-Hohenberg with $A=-(1+\partial_x^2)^2+\nu$ and $f(u)=-u^3$.
We can also treat a second example with  $A=- \partial_x^4+\nu$ where, in contrast to Swift-Hohenberg, the bifurcating mode is forced directly.

\begin{theorem}
\label{thm:stab}
Suppose~\eqref{e:SPDE} generates a RDS in $L^2_\rho$ and let $z$ be the unique stationary solution from Subsection~\ref{suse:SS}.
Under assumption~\ref{a:stab}.
If $\mathbb{E} z^2 > \eta / C_\delta$ then
the random attractor in $L^2_\rho$ is a single point given by the stationary solution~$z$.
\end{theorem}
\begin{proof}
\begin{equation*}
\begin{split}
 \frac12 \partial_t \|v\|^2
 & =
 \langle v,Av \rangle_{L^2_\rho}
 +  \langle f(v+z)-f(z), f(v)  \rangle_{L^2_\rho} \\
& \leq \eta \|v\|^2_{L^2_\rho} - \delta z^2 \|v\|^2_{L^2_\rho}
\end{split}
\end{equation*}

Thus by Gr\"onwall's inequality
\[
\|v(t)\|^2_{L^2_\rho} \leq \|v(0)\|^2_{L^2_\rho} \exp\{ \eta t - \delta  \int_0^t z^2 ds\}
\]
Now the claim follows by Birkhoff's theorem, as
\[ \frac1t \int_0^t z^2 ds \to \mathbb{E} z^2
\quad\text{for}\quad
t\to \infty\;.
\]
\end{proof}
We can now use this result to determine the regime, where the stationary solution
of SH is globally stable and thus stabilization sets is.
We need to calculate the expected value $\mathbb{E} z^2$. This will be done via Fokker-Planck in Section~\ref{suse:FP}.
\begin{remark}
 Let us remark that on bounded domains is is possible to study also local stability of the stationary solution in $H^1$, which holds true under a weaker condition.
 But on  unbounded domains this fails, as we cannot bound the remaining term $z \int \rho v v_x^2 dx$ by powers of the $H^1_\rho$-norm.
\end{remark}

Let us finally state a straightforward generalization to general additive noise.
Consider the following SPDE
\begin{equation}
 \label{e:SPDEgen}
 du= [Au + f(u)]dt + dW,
\end{equation}
for some general $Q$-Wiener process $W$.

\begin{assumption} 
\label{a:stabgen}
Suppose that~\eqref{e:SPDEgen} generates a RDS in $L^2_\rho$ with sufficiently smooth solutions.
Furthermore, let $Z$ be a stationary ergodic solution of 
\begin{equation*}
 dZ= [AZ + f(Z)]dt + dW,
\end{equation*}
and suppose that 
there is a  constant $C_\delta>0$ such that for all $v\in L^4_\rho$   
\[
	\langle f(v+Z)-f(Z), v  \rangle_{L^2_\rho} \leq  - C_\delta \|Zv\|^2_{L^2_\rho}.
\]
\end{assumption}
We remark without proof that this assumption is true in case of Swift-Hohenberg with the stable cubic $f(u)=-u^3$.
To generate a RDS one needs some regularity of the Wiener process $W$.

Define for the random variable $Z$ the maximum of the numerical range by
\begin{equation*} 
 \Lambda(Z) = \sup_{\|v\|^2_{L^2_\rho}=1} \{ \langle v, (A-C_\delta Z^2) v\rangle_{L^2_\rho} \}\;.
\end{equation*}

\begin{theorem}
\label{thm:genstab}
Under Assumption~\ref{a:stabgen}, if $\mathbb{E} \Lambda(Z) < 0$, then
the random attractor in $L^2_\rho$ is a single point given by the stationary solution $Z$.
\end{theorem}

The proof is similar to Theorem~\ref{thm:stab} using Gr\"onwall's lemma and Birkhoff's theorem.


\subsection{Fokker-Planck}
\label{suse:FP}

In the setting of Theorem~\ref{thm:stab} we can explicitly calculate the expected value $\mathbb{E}z^2$. 
The density $p$ of the stationary process $z$  corresponding to the SDE \eqref{e:SDE} solves the Fokker-Planck equation
\[
\frac12\sigma^2 p'' + [(\zeta^3-\lambda \zeta) p]' = 0\;.
\]
This has the explicit solution
\[
p(\zeta) = \frac1{C_N} \exp\{ - (\zeta^4 -2\lambda   \zeta^2)/ 2\sigma^2 \},
\]
with normalization constant
\[
C_N  = \int_\mathbb{R} \exp\{ - (\zeta^4 -2\lambda \zeta^2)/ 2\sigma^2 \} d\zeta \;.
\]
Thus we obtain
\begin{equation}
 \label{e:Ez2}
\mathbb{E} z^2 =  \frac{ \int_\mathbb{R} \zeta^2  \exp\{ - (\zeta^4 -2\lambda \zeta^2)/ 2\sigma^2 \} d\zeta }{\int_\mathbb{R} \exp\{ - (\zeta^4 -2\lambda \zeta^2)/ 2\sigma^2 \} d\zeta} \;.
\end{equation}
In the following, in our examples, we evaluate \eqref{e:Ez2} either asymptotically for small $\sigma$ or numerically for medium range $\sigma$.
We expect that  in the example of Swift-Hohenberg with $\lambda<0$, the point $0$ is deterministically stable in \eqref{e:SDE} with exponential rate $\lambda$.
Thus we  expect $\mathbb{E} z^2 \sim \sigma^2/|\lambda| $ for $\sigma$ small, which is the typical scaling of an OU-process.
We will treat this in our first example and quantify the amount of noise necessary to stabilize the equation, i.e.\ to destroy the random attractor so that it collapses to a single point.

In contrast to that, for $\lambda>0$ we expect that the limit $\lim_{\sigma\to 0} \mathbb{E} z^2$ exists and it is a $\lambda$-dependent positive constant.
This is due to the fact that solutions of~\eqref{e:SDE}  for small $\sigma$ concentrate around the two deterministic points $\pm\sqrt{\lambda}$. 
But we will see later in our second example with $A=-\partial_x^4+\nu$ that this is still not sufficient for proving stabilization for arbitrarily small noise.
Again, we need the noise to be strong enough. 


\subsection{Result for Swift-Hohenberg}

In our examples we have to check the requirements for Theorem~\ref{thm:stab} and Assumption~\ref{a:stab}.
For Swift-Hohenberg, Lemma~\ref{lem:spec} tells us that 
\[
\eta = \nu +\mathcal{O}(c^2)\quad \text{and} \quad \lambda = -1+\nu
\]
if the constant  $c$ in the weight $\rho$ is small enough (recall that $\rho \rightarrow 1$ for $c
\rightarrow 0$).

Lemma~\ref{lem:cubic} yields  with $\delta
\in \left( 0, \frac{1}{4} \right)$
\[
 C_{\delta} = 3 - \frac{9}{4 (1 - \delta)} = \frac{3}{4} +\mathcal{O} (\delta)\;.
\]
Our first result is the following:
\begin{prop}
Consider
the setting of Swift-Hohenberg, as in \eqref{e:SPDE}, with $\nu\in(0,1)$ and suppose that 
\[
\nu < \frac{3}{2} \sigma^2 \quad \text{and} \quad \sigma \ll 1-\nu\;.
\]
Then we can choose $0<c\ll1$ in the weight and the constant in the nonlinear estimate $0<\delta\ll1$ both sufficiently small
such that  stabilization holds in $L^2_\rho$. 
\end{prop}
\begin{proof}
We first notice that for $c$ and $\delta$ sufficiently small
by Theorem~\ref{thm:stab} stabilization holds in case $\nu<\frac34\mathbb{E}z^2$.
Using then the change of variables $y = \frac{1}{\sqrt{1 - \nu}} z$, 
the
condition for stabilization becomes $\mathbb{E}y^2 > \frac{4}{3} 
\frac{\nu}{1 - \nu}$. Let us define, for ease of notation, $\sigma_{\nu} = \frac{\sigma}{1
- \nu}$. From \eqref{e:Ez2} we obtain
\begin{equation*}
  \begin{split}
  \mathbb{E}y^2 & =   {\int_{\mathbb{R}} y^2 e^{- \frac{1}{8
  \sigma_{\nu}^2}  (y^2 + 1)^2} d y}  \Bigg/  {\int_{\mathbb{R}} e^{- \frac{1}{8
  \sigma_{\nu}^2}  (y^2 + 1)^2} d y}\\
  & =  {e^{- \frac{1}{8 \sigma_{\nu}^2}}  \int_{\mathbb{R}} y^2 e^{-
  \frac{1}{8 \sigma_{\nu}^2} y^2 (y^2 + 2)} d y} \Bigg/ {e^{- \frac{1}{8
  \sigma_{\nu}^2}}  \int_{\mathbb{R}} e^{- \frac{1}{8 \sigma_{\nu}^2} y^2
  (y^2 + 2)} d y}\\
  & = {\sigma_{\nu}^2  \int_{\mathbb{R}} w^2
  e^{- \frac{1}{8} w^2 (\sigma_{\nu}^2 w^2 + 2)} d w}\Bigg/{\int_{\mathbb{R}}
  e^{- \frac{1}{8} w^2 (\sigma_{\nu}^2 w^2 + 2)} d w}.
  \end{split}
\end{equation*}
Now we can observe that $\sigma_{\nu}^2 z^2 + 2 \approx 2$ for $\sigma_{\nu}^2
z^2 \ll 2$ that is $z^2 \ll 2 \sigma_{\nu}^{- 2}$. So let us fix $\alpha \in
(0, 1)$ to cut out the exponentially small tails in the integral:
\begin{equation*}
  \begin{split}
 \int_{\mathbb{R}} w^2
  e^{- \frac{1}{8} w^2 (\sigma_{\nu}^2 w^2 + 2)} d w 
  & = \int_{- \sigma_{\nu}^{-
  \alpha}}^{\sigma_{\nu}^{- \alpha}} w^2 e^{- \frac{1}{4} w^2 (1 + O
  (\sigma_{\nu}^{2 - 2 \alpha}))} d w + O (e^{- c \sigma_{\nu}^{- 2
  \alpha}}) \\&=\int_{\mathbb{R}} w^2 e^{- \frac{1}{4} w^2 (1
  + O (\sigma_{\nu}^{2 - 2 \alpha}))} d w + O (e^{- c \sigma_{\nu}^{- 2
  \alpha}})
   \\&= 4 \sqrt{\pi}  (1 + O (\sigma_{\nu}^{2 - 2
  \alpha}))^{- \frac{3}{2}} + O (e^{- c \sigma_{\nu}^{- 2 \alpha}})\;,
  \end{split}
\end{equation*}
where we calculated explicitly the Gaussian integral.
Similarly for the denominator
\[ \int_{\mathbb{R}}
  e^{- \frac{1}{8} w^2 (\sigma_{\nu}^2 w^2 + 2)} d w  
  = 2
  \sqrt{\pi}  (1 + O (\sigma_{\nu}^{2 - 2 \alpha}))^{- \frac{1}{2}} + O (e^{-
  c \sigma_{\nu}^{- 2 \alpha}})\;.
  \]
  Thus we obtain 
\begin{equation*}
  \mathbb{E}y^2 
  = 2 \sigma_{\nu}^2  \frac{(1 + O (\sigma_{\nu}^{2 - 2 \alpha}))^{-
  \frac{3}{2}} + O (e^{- c \sigma_{\nu}^{- 2 \alpha}})}{(1 + O
  (\sigma_{\nu}^{2 - 2 \alpha}))^{- \frac{1}{2}} + O (e^{- c \sigma_{\nu}^{- 2
  \alpha}})}
 = 2 \sigma_{\nu}^2 (1 + O (\sigma_{\nu}^{2 - 2 \alpha})) .
\end{equation*}
Now we can re-substitute  and 
rewrite the condition for stability as 
\[ \nu (1 - \nu) < \frac{3}{2} \sigma^2  (1 + O (\sigma_{\nu}^{2 - 2 \alpha}))  .
\]
 Note that the previous conditions are only satisfied if also $\nu$ is small, thus we can simplify to  $ \nu < \frac{3}{2} \sigma^2$.
\end{proof}

If we compare this result with the one in~\cite{mohblokle13}, we have that for $\nu$ close to 0, 
the two are quite similar, requiring a noise strength such that $\nu < \frac{3}{2} \sigma^2$.
But this is just the asymptotic evaluation of our result. In Figure~\ref{fig:sh1} we see various curves 
of the function $\sigma\mapsto \frac34\mathbb{E}z^2/\nu$ corresponding to the stability condition for different~$\nu$. 
We see that once $\sigma$ is sufficiently large, stabilization sets in.
In Figure~\ref{fig:sh2} we see the various values  of the noise-strength $\sigma$  after which stabilization sets in as a function of $\nu$.

\begin{figure}[htbp!]
\centering
\hspace*{-.09\textwidth}
\scalebox{.675}{
\setlength{\unitlength}{1pt}
\begin{picture}(0,0)
\includegraphics{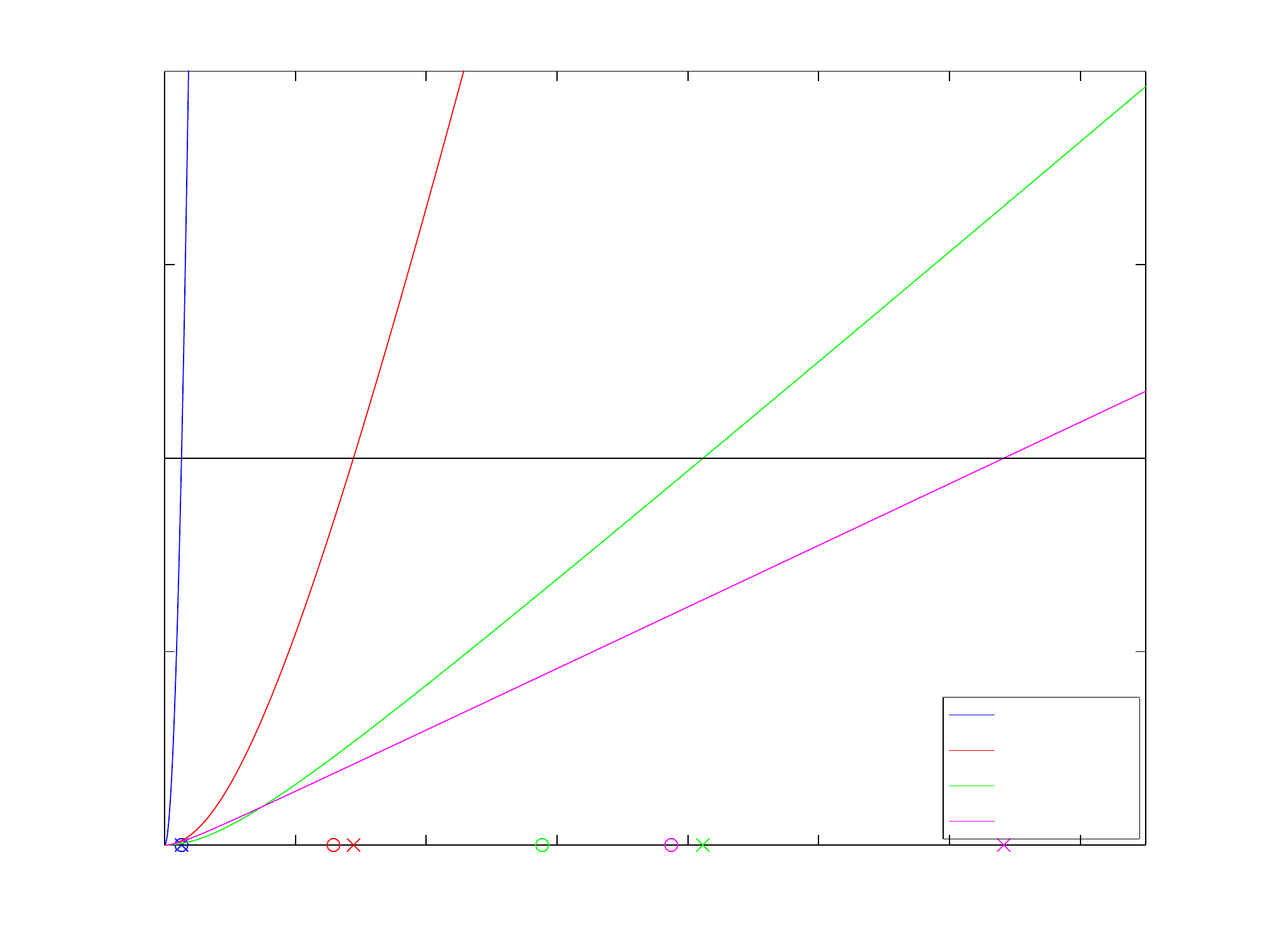}
\end{picture}%
\begin{picture}(576,433)(0,0)
\fontsize{10}{0}
\selectfont\put(74.88,43.5189){\makebox(0,0)[t]{\textcolor[rgb]{0,0,0}{{0}}}}
\fontsize{10}{0}
\selectfont\put(134.4,43.5189){\makebox(0,0)[t]{\textcolor[rgb]{0,0,0}{{0.2}}}}
\fontsize{10}{0}
\selectfont\put(193.92,43.5189){\makebox(0,0)[t]{\textcolor[rgb]{0,0,0}{{0.4}}}}
\fontsize{10}{0}
\selectfont\put(253.44,43.5189){\makebox(0,0)[t]{\textcolor[rgb]{0,0,0}{{0.6}}}}
\fontsize{10}{0}
\selectfont\put(312.96,43.5189){\makebox(0,0)[t]{\textcolor[rgb]{0,0,0}{{0.8}}}}
\fontsize{10}{0}
\selectfont\put(372.48,43.5189){\makebox(0,0)[t]{\textcolor[rgb]{0,0,0}{{1}}}}
\fontsize{10}{0}
\selectfont\put(432,43.5189){\makebox(0,0)[t]{\textcolor[rgb]{0,0,0}{{1.2}}}}
\fontsize{10}{0}
\selectfont\put(491.52,43.5189){\makebox(0,0)[t]{\textcolor[rgb]{0,0,0}{{1.4}}}}
\fontsize{10}{0}
\selectfont\put(69.8755,48.52){\makebox(0,0)[r]{\textcolor[rgb]{0,0,0}{{0}}}}
\fontsize{10}{0}
\selectfont\put(69.8755,136.54){\makebox(0,0)[r]{\textcolor[rgb]{0,0,0}{{0.5}}}}
\fontsize{10}{0}
\selectfont\put(69.8755,224.56){\makebox(0,0)[r]{\textcolor[rgb]{0,0,0}{{1}}}}
\fontsize{10}{0}
\selectfont\put(69.8755,312.58){\makebox(0,0)[r]{\textcolor[rgb]{0,0,0}{{1.5}}}}
\fontsize{10}{0}
\selectfont\put(69.8755,400.6){\makebox(0,0)[r]{\textcolor[rgb]{0,0,0}{{2}}}}
\fontsize{10}{0}
\selectfont\put(298.08,30.5189){\makebox(0,0)[t]{\textcolor[rgb]{0,0,0}{{$\sigma$}}}}
\fontsize{10}{0}
\selectfont\put(50.8755,224.56){\rotatebox{90}{\makebox(0,0)[b]{\textcolor[rgb]{0,0,0}{{$\frac{3}{4\nu}\mathbb{E}z^2$}}}}}
\fontsize{10}{0}
\selectfont\put(455.104,107.673){\makebox(0,0)[l]{\textcolor[rgb]{0,0,0}{{$\quad\nu=0.001$}}}}
\fontsize{10}{0}
\selectfont\put(455.104,91.5562){\makebox(0,0)[l]{\textcolor[rgb]{0,0,0}{{$\quad\nu=0.1$}}}}
\fontsize{10}{0}
\selectfont\put(455.104,75.4397){\makebox(0,0)[l]{\textcolor[rgb]{0,0,0}{{$\quad\nu=0.5$}}}}
\fontsize{10}{0}
\selectfont\put(455.104,59.3233){\makebox(0,0)[l]{\textcolor[rgb]{0,0,0}{{$\quad\nu=0.9$}}}}
\end{picture}
}
\caption{The function $\sigma\mapsto \frac34\mathbb{E}z^2/\nu$ for various values of $\nu$. The crosses indicate the point where the stability condition is an equality, 
 while the circles indicate the prediction by the asymptotic formula given by $\nu = \frac{3}{2} \sigma^2$.}
 \label{fig:sh1}
\end{figure}
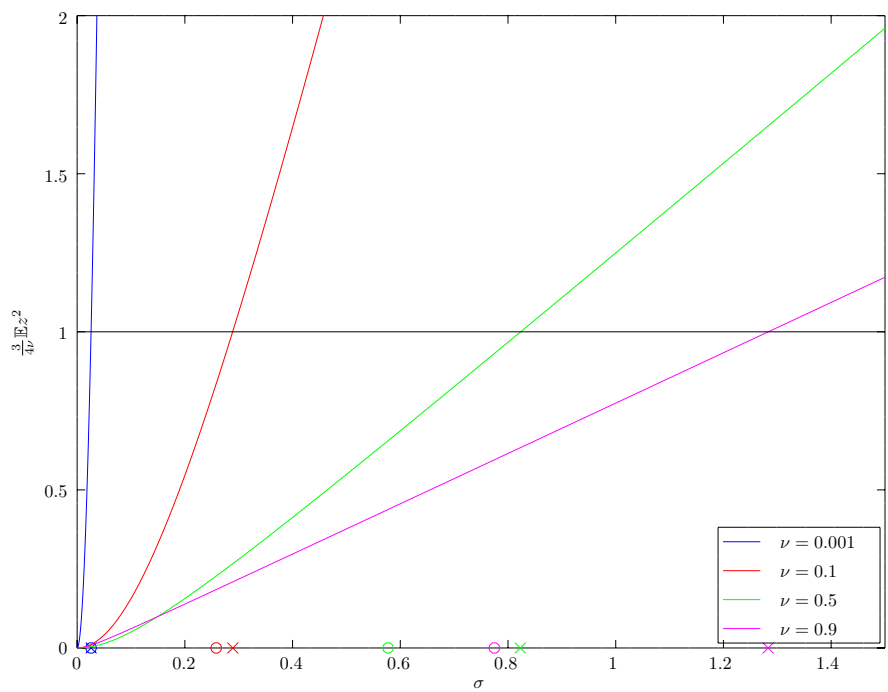

\begin{figure}[htb!]
\centering
\hspace*{-.09\textwidth}
\scalebox{.675}{
\setlength{\unitlength}{1pt}
\begin{picture}(0,0)
\includegraphics{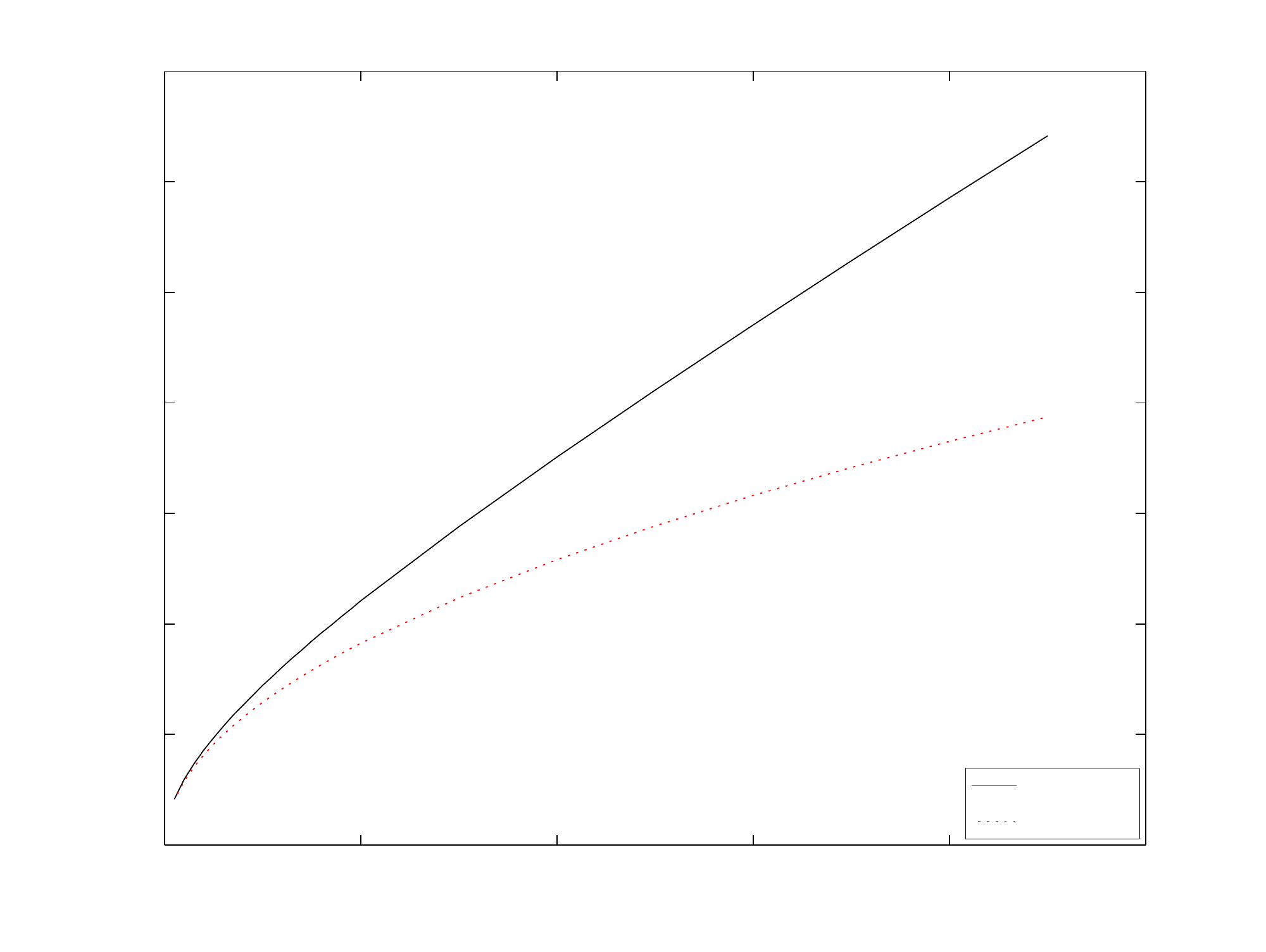}
\end{picture}%
\begin{picture}(576,433)(0,0)
\fontsize{10}{0}
\selectfont\put(74.88,43.5189){\makebox(0,0)[t]{\textcolor[rgb]{0,0,0}{{0}}}}
\fontsize{10}{0}
\selectfont\put(164.16,43.5189){\makebox(0,0)[t]{\textcolor[rgb]{0,0,0}{{0.2}}}}
\fontsize{10}{0}
\selectfont\put(253.44,43.5189){\makebox(0,0)[t]{\textcolor[rgb]{0,0,0}{{0.4}}}}
\fontsize{10}{0}
\selectfont\put(342.72,43.5189){\makebox(0,0)[t]{\textcolor[rgb]{0,0,0}{{0.6}}}}
\fontsize{10}{0}
\selectfont\put(432,43.5189){\makebox(0,0)[t]{\textcolor[rgb]{0,0,0}{{0.8}}}}
\fontsize{10}{0}
\selectfont\put(521.28,43.5189){\makebox(0,0)[t]{\textcolor[rgb]{0,0,0}{{1}}}}
\fontsize{10}{0}
\selectfont\put(69.8755,48.52){\makebox(0,0)[r]{\textcolor[rgb]{0,0,0}{{0}}}}
\fontsize{10}{0}
\selectfont\put(69.8755,98.8172){\makebox(0,0)[r]{\textcolor[rgb]{0,0,0}{{0.2}}}}
\fontsize{10}{0}
\selectfont\put(69.8755,149.114){\makebox(0,0)[r]{\textcolor[rgb]{0,0,0}{{0.4}}}}
\fontsize{10}{0}
\selectfont\put(69.8755,199.411){\makebox(0,0)[r]{\textcolor[rgb]{0,0,0}{{0.6}}}}
\fontsize{10}{0}
\selectfont\put(69.8755,249.709){\makebox(0,0)[r]{\textcolor[rgb]{0,0,0}{{0.8}}}}
\fontsize{10}{0}
\selectfont\put(69.8755,300.006){\makebox(0,0)[r]{\textcolor[rgb]{0,0,0}{{1}}}}
\fontsize{10}{0}
\selectfont\put(69.8755,350.303){\makebox(0,0)[r]{\textcolor[rgb]{0,0,0}{{1.2}}}}
\fontsize{10}{0}
\selectfont\put(69.8755,400.6){\makebox(0,0)[r]{\textcolor[rgb]{0,0,0}{{1.4}}}}
\fontsize{10}{0}
\selectfont\put(298.08,30.5189){\makebox(0,0)[t]{\textcolor[rgb]{0,0,0}{{$\nu$}}}}
\fontsize{10}{0}
\selectfont\put(50.8755,224.56){\rotatebox{90}{\makebox(0,0)[b]{\textcolor[rgb]{0,0,0}{{$\sigma$}}}}}
\fontsize{10}{0}
\selectfont\put(465.39,75.4397){\makebox(0,0)[l]{\textcolor[rgb]{0,0,0}{{actual}}}}
\fontsize{10}{0}
\selectfont\put(465.39,59.3232){\makebox(0,0)[l]{\textcolor[rgb]{0,0,0}{{estimated}}}}
\end{picture}
}
 \caption{A plot of the noise-strength $\sigma$  after which stabilization sets in as a function of $\nu$.}
 \label{fig:sh2}
\end{figure}


\subsection{A second example}

In this second example we consider the setting of~\eqref{e:SPDE}, but now with $A=-\partial_x^4 +\nu$.
Similar to Lemma~\ref{lem:spec}, we obtain in this case 
\[\eta = \nu + O_+(c^2)
\quad \text{and now } \lambda = \nu.
\] 
So we need in Theorem~\ref{thm:stab} that $\frac{3}{4} \mathbb{E}z^2 > \nu$.
Substituting $y (t) = \nu^{- 1 / 2} z (t\nu^{- 1})$ we have 
\[ 
	d y = (y - y^3) d t + \frac{\sigma}{\nu} d B,
\]
and the stability condition changes to $\mathbb{E}y^2 > \frac{4}{3}$.

Let us define $\sigma_{\nu} = \frac{\sigma}{\nu}$, as well as
\[ 
	\mathcal{E} (y) = \frac{1}{4} y^4 - \frac{1}{2} y^2 + \frac{1}{4} =
   \frac{1}{4}  (y - 1)^2 (y + 1)^2 . 
\]
Then we have, for the second moment:
\begin{equation*}
  \begin{split}
  \mathbb{E}y^2 & = \frac{\int_{\mathbb{R}} y^2 e^{-\mathcal{E} (y) / 2
  \sigma^2_{\nu}} d y}{\int_{\mathbb{R}} e^{-\mathcal{E} (y) / 2
  \sigma^2_{\nu}} d y}
   = \frac{2 \int_0^{+ \infty} y^2 e^{- \frac{1}{8 \sigma_{\nu}^2}  (y^2 -
  1)^2} d y}{2 \int_0^{+ \infty} e^{- \frac{1}{8 \sigma_{\nu}^2}  (y^2 -
  1)^2} d y}\\
  & = \frac{\int_{- \sigma_{\nu}^{- 1}}^{+
  \infty} (\sigma_{\nu} y + 1)^2 e^{- \frac{1}{8} y^2 (\sigma_{\nu} y + 2)^2}
  d y}{\int_{- \sigma_{\nu}^{- 1}}^{+ \infty} e^{- \frac{1}{8} y^2
  (\sigma_{\nu} y + 2)^2} d y}.
  \end{split}
\end{equation*}
Now $y \sigma_{\nu} + 2 \approx 2$ for $y \sigma_{\nu} \ll 1$ (i.e.,  $y \ll
\sigma_{\nu}^{- 1}$). We fix $\alpha \in (0, 1)$ to control the small exponential tails:
\begin{equation*}
  \begin{split}
  \mathbb{E}y^2 & = \frac{\int_{- \sigma_{\nu}^{- \alpha}}^{\sigma_{\nu}^{-
  \alpha}} (\sigma_{\nu} y + 1)^2 e^{- \frac{1}{2} y^2 (1 + O (\sigma_{\nu}^{1
  - \alpha}))^2} d y + O (e^{- c \sigma_{\nu}^{- \alpha}})}{\int_{-
  \sigma_{\nu}^{- \alpha}}^{\sigma_{\nu}^{- \alpha}} e^{- \frac{1}{2} y^2 (1 +
  O (\sigma_{\nu}^{1 - \alpha}))^2} d y + O (e^{- c \sigma_{\nu}^{-
  \alpha}})}\\
  & = \frac{\int_{\mathbb{R}} (\sigma_{\nu}^2
  y^2 + 1) e^{- \frac{1}{2} y^2 (1 + O (\sigma_{\nu}^{1 - \alpha}))^2} d
  y + O (e^{- c \sigma_{\nu}^{- \alpha}})}{\int_{\mathbb{R}} e^{- \frac{1}{2}
  y^2 (1 + O (\sigma_{\nu}^{1 - \alpha}))^2} d y + O (e^{- c
  \sigma_{\nu}^{- \alpha}})}\\
  & = \frac{\sqrt{2 \pi}  (1 + O (\sigma_{\nu}^{1 - \alpha}))^{- 1} +
  \sigma_{\nu}^2  \sqrt{2 \pi}  (1 + O (\sigma_{\nu}^{1 - \alpha}))^{- 3} + O
  (e^{- c \sigma_{\nu}^{- \alpha}})}{\sqrt{2 \pi}  (1 + O (\sigma_{\nu}^{1 -
  \alpha}))^{- 1} + O (e^{- c \sigma_{\nu}^{- \alpha}})}\\
  & = \frac{1 + \sigma_{\nu}^2 + O (\sigma_{\nu}^{1 - \alpha}) + O (e^{- c
  \sigma_{\nu}^{- \alpha}})}{1 + O (\sigma_{\nu}^{1 - \alpha}) + O (e^{- c
  \sigma_{\nu}^{- \alpha}})},
  \end{split}
\end{equation*}
where we took advantage of $2 \sigma_{\nu}$ being an odd function.
Now it is easy to show that stabilization sets in for sufficiently large noise strength, but our estimate is not optimal here, because we cannot provide a sharp control of the error term.

Nevertheless, 
we conjecture that $\E y^2 = 1 + \sigma_{\nu}^2$ and thus for the condition on stabilization we obtain 
 $\sigma^2 > \nu^2/ 3$. 
 
 In Figure~\ref{fig:2exp1} we see various curves 
of the function $\sigma\mapsto \frac34\mathbb{E}z^2/\nu$ corresponding to the stability condition for different $\nu$. 
We see again that once $\sigma$ is sufficiently large, stabilization sets in.
In Figure~\ref{fig:sh2} we see the various values of the noise-strength $\sigma$  after which stabilization sets in as a function of $\nu$.

\begin{figure}[htbp!]
\centering
\hspace*{-.09\textwidth}
\scalebox{.675}{
\setlength{\unitlength}{1pt}
\begin{picture}(0,0)
\includegraphics{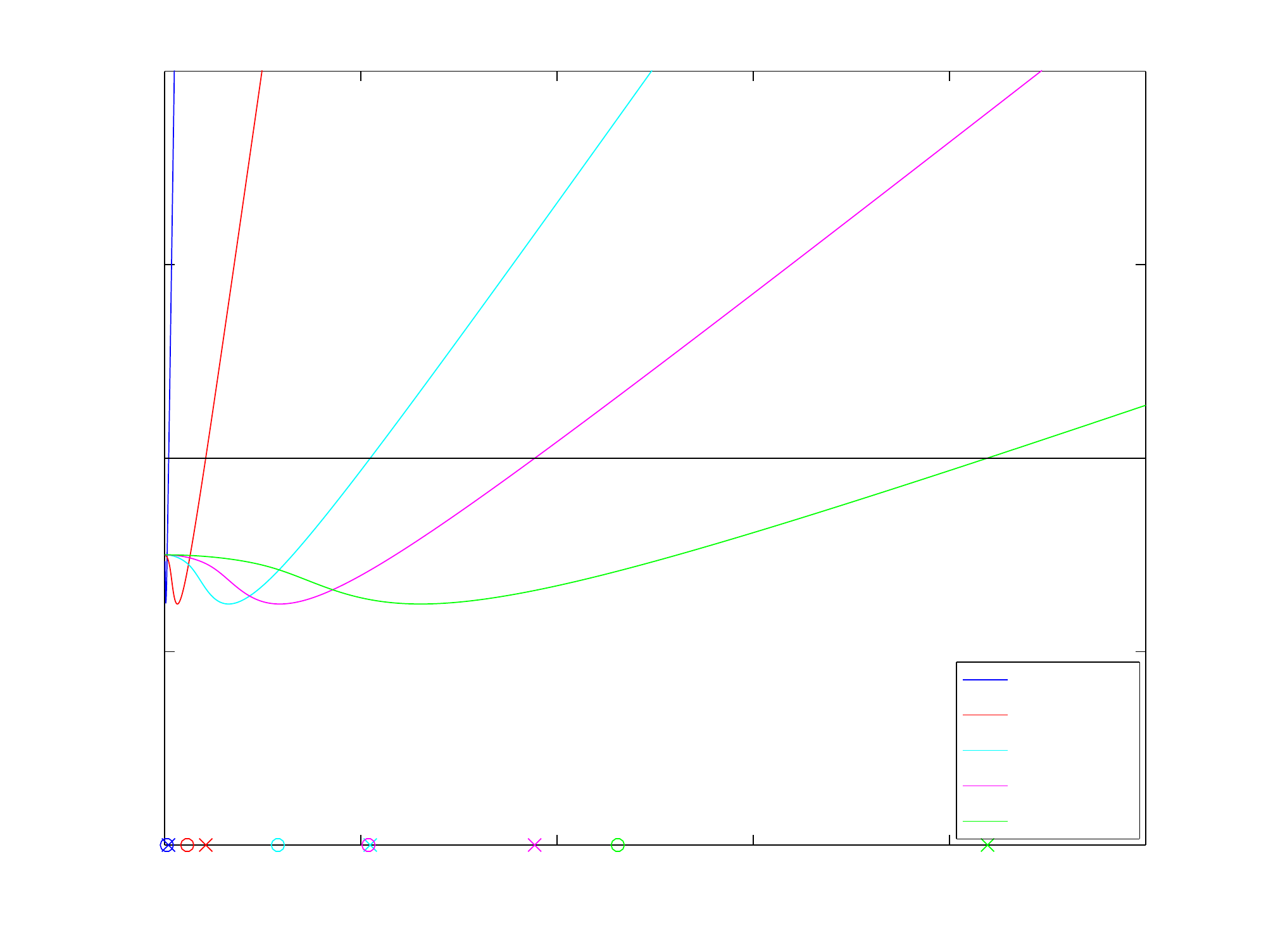}
\end{picture}%
\begin{picture}(576,433)(0,0)
\fontsize{10}{0}
\selectfont\put(74.88,43.5189){\makebox(0,0)[t]{\textcolor[rgb]{0,0,0}{{0}}}}
\fontsize{10}{0}
\selectfont\put(164.16,43.5189){\makebox(0,0)[t]{\textcolor[rgb]{0,0,0}{{0.5}}}}
\fontsize{10}{0}
\selectfont\put(253.44,43.5189){\makebox(0,0)[t]{\textcolor[rgb]{0,0,0}{{1}}}}
\fontsize{10}{0}
\selectfont\put(342.72,43.5189){\makebox(0,0)[t]{\textcolor[rgb]{0,0,0}{{1.5}}}}
\fontsize{10}{0}
\selectfont\put(432,43.5189){\makebox(0,0)[t]{\textcolor[rgb]{0,0,0}{{2}}}}
\fontsize{10}{0}
\selectfont\put(521.28,43.5189){\makebox(0,0)[t]{\textcolor[rgb]{0,0,0}{{2.5}}}}
\fontsize{10}{0}
\selectfont\put(69.8755,48.52){\makebox(0,0)[r]{\textcolor[rgb]{0,0,0}{{0}}}}
\fontsize{10}{0}
\selectfont\put(69.8755,136.54){\makebox(0,0)[r]{\textcolor[rgb]{0,0,0}{{0.5}}}}
\fontsize{10}{0}
\selectfont\put(69.8755,224.56){\makebox(0,0)[r]{\textcolor[rgb]{0,0,0}{{1}}}}
\fontsize{10}{0}
\selectfont\put(69.8755,312.58){\makebox(0,0)[r]{\textcolor[rgb]{0,0,0}{{1.5}}}}
\fontsize{10}{0}
\selectfont\put(69.8755,400.6){\makebox(0,0)[r]{\textcolor[rgb]{0,0,0}{{2}}}}
\fontsize{10}{0}
\selectfont\put(298.08,30.5189){\makebox(0,0)[t]{\textcolor[rgb]{0,0,0}{{$\sigma$}}}}
\fontsize{10}{0}
\selectfont\put(50.8755,224.56){\rotatebox{90}{\makebox(0,0)[b]{\textcolor[rgb]{0,0,0}{{$\frac{3}{4\nu}\mathbb{E}z^2$}}}}}
\fontsize{10}{0}
\selectfont\put(461.276,123.789){\makebox(0,0)[l]{\textcolor[rgb]{0,0,0}{{$\quad\nu=0.01$}}}}
\fontsize{10}{0}
\selectfont\put(461.276,107.673){\makebox(0,0)[l]{\textcolor[rgb]{0,0,0}{{$\quad\nu=0.1$}}}}
\fontsize{10}{0}
\selectfont\put(461.276,91.5562){\makebox(0,0)[l]{\textcolor[rgb]{0,0,0}{{$\quad\nu=0.5$}}}}
\fontsize{10}{0}
\selectfont\put(461.276,75.4397){\makebox(0,0)[l]{\textcolor[rgb]{0,0,0}{{$\quad\nu=0.9$}}}}
\fontsize{10}{0}
\selectfont\put(461.276,59.3233){\makebox(0,0)[l]{\textcolor[rgb]{0,0,0}{{$\quad\nu=2$}}}}
\end{picture}
}
 \caption{The function $\sigma\mapsto \frac34\mathbb{E}z^2/\nu$ for various values of $\nu$. The crosses indicate the point where the stability condition is an equality, 
 while the circles indicate the prediction by the asymptotic formula conjectured to be  $\nu = \sqrt3 \sigma$.}
 \label{fig:2exp1}
\end{figure}

\begin{figure}[htbp!]
\centering
\hspace*{-.09\textwidth}
\scalebox{.675}{
\setlength{\unitlength}{1pt}
\begin{picture}(0,0)
\includegraphics{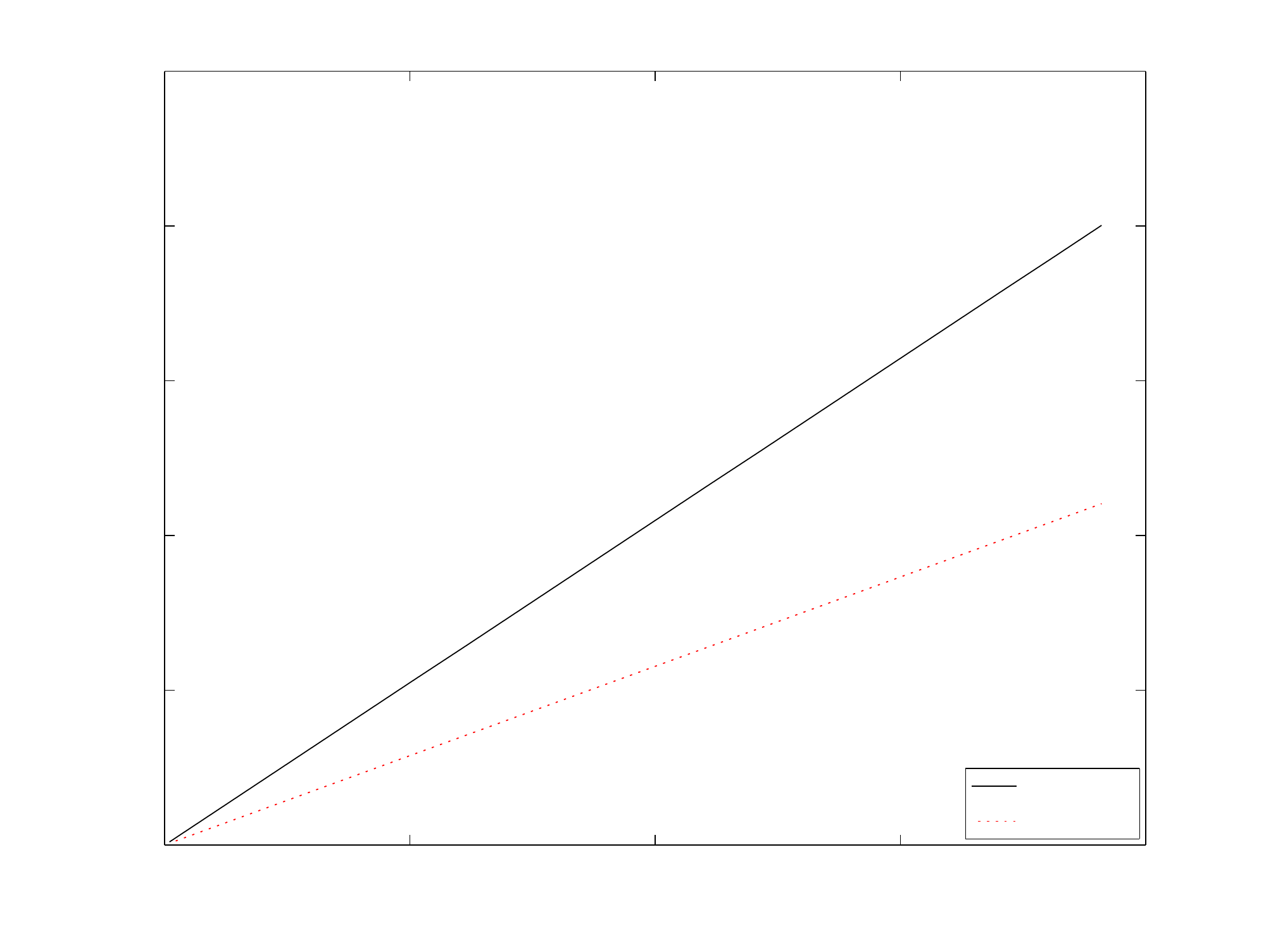}
\end{picture}%
\begin{picture}(576,433)(0,0)
\fontsize{10}{0}
\selectfont\put(74.88,43.5189){\makebox(0,0)[t]{\textcolor[rgb]{0,0,0}{{0}}}}
\fontsize{10}{0}
\selectfont\put(186.48,43.5189){\makebox(0,0)[t]{\textcolor[rgb]{0,0,0}{{0.5}}}}
\fontsize{10}{0}
\selectfont\put(298.08,43.5189){\makebox(0,0)[t]{\textcolor[rgb]{0,0,0}{{1}}}}
\fontsize{10}{0}
\selectfont\put(409.68,43.5189){\makebox(0,0)[t]{\textcolor[rgb]{0,0,0}{{1.5}}}}
\fontsize{10}{0}
\selectfont\put(521.28,43.5189){\makebox(0,0)[t]{\textcolor[rgb]{0,0,0}{{2}}}}
\fontsize{10}{0}
\selectfont\put(69.8755,48.52){\makebox(0,0)[r]{\textcolor[rgb]{0,0,0}{{0}}}}
\fontsize{10}{0}
\selectfont\put(69.8755,118.936){\makebox(0,0)[r]{\textcolor[rgb]{0,0,0}{{0.5}}}}
\fontsize{10}{0}
\selectfont\put(69.8755,189.352){\makebox(0,0)[r]{\textcolor[rgb]{0,0,0}{{1}}}}
\fontsize{10}{0}
\selectfont\put(69.8755,259.768){\makebox(0,0)[r]{\textcolor[rgb]{0,0,0}{{1.5}}}}
\fontsize{10}{0}
\selectfont\put(69.8755,330.184){\makebox(0,0)[r]{\textcolor[rgb]{0,0,0}{{2}}}}
\fontsize{10}{0}
\selectfont\put(69.8755,400.6){\makebox(0,0)[r]{\textcolor[rgb]{0,0,0}{{2.5}}}}
\fontsize{10}{0}
\selectfont\put(298.08,30.5189){\makebox(0,0)[t]{\textcolor[rgb]{0,0,0}{{$\nu$}}}}
\fontsize{10}{0}
\selectfont\put(50.8755,224.56){\rotatebox{90}{\makebox(0,0)[b]{\textcolor[rgb]{0,0,0}{{$\sigma$}}}}}
\fontsize{10}{0}
\selectfont\put(465.39,75.3757){\makebox(0,0)[l]{\textcolor[rgb]{0,0,0}{{actual}}}}
\fontsize{10}{0}
\selectfont\put(465.39,59.2976){\makebox(0,0)[l]{\textcolor[rgb]{0,0,0}{{estimated}}}}
\end{picture}
}
 \caption{A plot of the noise-strength $\sigma$  after which stabilization sets in as a function of $\nu$.}
 \label{fig:2exp2}
\end{figure}

 \subsection{Comments on further examples}
 
Many examples we tried have a similar result than the two results presented here.
For sufficiently large noise strength one obtains stabilization.
This is the case, when we replace the cubic with a stable polynomial of higher odd degree like $f(u)= - u|u|^{2p}$, $p\in \mathbb{N}$. 
 
If we consider $\partial_t u = -(\partial_x^2 +\mu)^2 u + \nu u - u^3 + \sigma \partial_t \beta$ we again obtain a result similar to the one for Swift-Hohenberg, 
that is stabilization occurs for small $\nu$.  The drawback in this setting is that the noise intensity $\sigma$ has to satisfy the following two conditions:
\[
	(\nu -\mu^2)(\frac{1}{3}\nu +\mu^2)< \sigma^2 \ll \nu-\mu^2.
\]

Finally, another example one could think about is the following: $\partial_t u = \partial_x^4 u +
\alpha u^2 - u^3$, but in this case our approach never worked well, because our
estimate is always a little bit off.


\section*{Acknowledgements}
L.B.\ and D.B.\ are supported by  DFG-funding BL535-9/2 ``Mehrskalenanalyse stochastischer partieller Differentialgleichungen (SPDEs)''.
M.Y.\ was
supported by the NSFC grants(No.11571125) and NCET-12-0204.
A visit of M.H. in Augsburg was supported by the ``Gastwissenschafterprogramm des bayrischen Staatsministeriums f\"ur Bildung und Kultus, Wissenschaft und Kunst''.

\FloatBarrier
\bibliographystyle{abbrv}
\bibliography{biblio}
\end{document}